\documentclass[a4paper,12pt]{article}

\usepackage[french]{babel}
\usepackage[utf8]{inputenc}
\usepackage[T1]{fontenc}
\usepackage{lmodern}
\usepackage{amsmath,amssymb,amsthm}
\usepackage{enumerate}
\usepackage{tikz-cd}
\usepackage[]{hyperref}

\DeclareFontFamily{OMX}{lmex}{}
\DeclareFontShape{OMX}{lmex}{m}{n}{<->lmex10}{}

\theoremstyle{plain}
\newtheorem{theo}{Théorème}[section]
\newtheorem{prop}[theo]{Proposition}
\newtheorem{propapp}[equation]{Proposition}
\newtheorem{conj}[theo]{Conjecture}
\newtheorem{coro}[theo]{Corollaire}
\newtheorem{lemm}[theo]{Lemme}
\newtheorem{lemmapp}[equation]{Lemme}
\theoremstyle{definition}
\newtheorem{defi}[theo]{Définition}
\newtheorem{defiapp}[equation]{Définition}
\theoremstyle{remark}
\newtheorem{rema}[theo]{Remarque}

\DeclareMathOperator{\im}{im}
\DeclareMathOperator{\rg}{rg}
\DeclareMathOperator{\card}{card}
\DeclareMathOperator{\supp}{supp}
\DeclareMathOperator{\Hom}{Hom}
\DeclareMathOperator{\Ext}{Ext}
\DeclareMathOperator{\Mod}{Mod}
\DeclareMathOperator{\Rep}{Rep}
\DeclareMathOperator{\Lie}{Lie}
\DeclareMathOperator{\Gal}{Gal}
\DeclareMathOperator{\detfr}{d\acute{e}t}
\newcommand{\llbrack}{[\![}
\newcommand{\rrbrack}{]\!]}
\newcommand{\dfn}{\overset{\text{déf}}{=}}
\newcommand{\iso}{\overset{\sim}{\longrightarrow}}
\newcommand{\osi}{\overset{\sim}{\longleftarrow}}
\newcommand{\h}{\overset{\mathrm{H}}{\cdot}}
\newcommand{\diff}{\mathrm{d}}
\newcommand{\un}{\underline{1}}
\newcommand{\Nbb}{\mathbb{N}}
\newcommand{\Z}{\mathbb{Z}}
\newcommand{\Q}{\mathbb{Q}}
\newcommand{\Fp}{\mathbb{F}_p}
\newcommand{\Zp}{\Z_p}
\newcommand{\Qp}{\Q_p}
\newcommand{\Qpbar}{\overline{\Qp}}
\newcommand{\Oe}{\mathcal{O}_{\!E}}
\newcommand{\pe}{\varpi_{\!E}}
\newcommand{\ke}{k_E}
\newcommand{\A}[1]{\Oe/\pe^{#1}\Oe}
\newcommand{\val}{\operatorname{val}_p}
\newcommand{\Nrm}{\operatorname{N}_{F/\Qp}}
\newcommand{\Res}{\operatorname{Res}_{F/\Qp}}
\newcommand{\For}{\mathcal{O}}
\newcommand{\GL}{\mathrm{GL}}
\newcommand{\GSp}{\mathrm{GSp}}
\newcommand{\epsa}{\varepsilon^{-1} \circ \alpha}
\newcommand{\oma}{\omega^{-1} \circ \alpha}
\newcommand{\epsth}{\cdot (\varepsilon^{-1} \circ \theta)}
\newcommand{\omth}{\cdot (\omega^{-1} \circ \theta)}
\newcommand{\N}{\widetilde{N}}
\newcommand{\Nz}{\widetilde{N}_0}
\newcommand{\Nwz}{N_{w,0}}
\newcommand{\Ns}[1]{N_{s_{#1} \dots s_1}}
\newcommand{\Nsz}[1]{N_{s_{#1} \dots s_1,0}}
\newcommand{\Plz}[1]{\Pi_{#1,0}}
\newcommand{\NT}{T^+ \ltimes \widetilde{N}_0}
\newcommand{\V}{\widetilde{V}}
\newcommand{\Vwz}{V_{w,0}}
\newcommand{\fin}[1]{#1\!\operatorname{-fin}}
\newcommand{\Homcont}{\Hom^\mathrm{cont}}
\newcommand{\HomGr}{\Homcont_\mathrm{Grp}}
\newcommand{\HomTp}[1]{\Hom_{A[T^+]}(A[T(F)],#1)_{\fin{T(F)}}}
\newcommand{\Hc}[1][\bullet]{\mathrm{H}^{#1}}
\newcommand{\Rc}[1][\bullet]{\mathrm{R}^{#1}}
\newcommand{\Ord}{\operatorname{Ord}_{B(F)}}
\newcommand{\HOrd}[1][\bullet]{\Hc[#1]\!\Ord}
\newcommand{\ROrd}[1][\bullet]{\Rc[#1]\!\Ord}
\newcommand{\HOrdQp}{\operatorname{H^1Ord}_{B(\Qp)}}
\newcommand{\Ind}{\operatorname{Ind}^{G(F)}_{B^-(F)}}
\newcommand{\IndQp}{\operatorname{Ind}^{G(\Qp)}_{B^-(\Qp)}}
\newcommand{\Indd}{\operatorname{Ind}^{\GL_2(\Qp)}_{B_2^-(\Qp)}}
\newcommand{\IndP}{\operatorname{Ind}^{G(F)}_{P^-(F)}}
\newcommand{\IndPQp}{\operatorname{Ind}^{G(\Qp)}_{P^-(\Qp)}}
\newcommand{\IndPa}{\operatorname{Ind}^{G(\Qp)}_{P_\alpha^-(\Qp)}}
\newcommand{\E}{\mathcal{E}}
\newcommand{\C}{\mathcal{C}}
\newcommand{\Clis}{\C^\infty}
\newcommand{\BB}{(B^-(F),B(F))}
\newcommand{\BG}{B^-(F) \backslash G(F)}
\newcommand{\BC}{B^-(F) \backslash C}
\newcommand{\BD}{B^-(F) \backslash D}
\newcommand{\w}{\dot{w}}
\newcommand{\wP}{\widetilde{w}_P}
\newcommand{\wPd}{\dot{\widetilde{w}}_P}
\newcommand{\WP}{\widetilde{W}_P}

\hypersetup{pdfinfo={
	Title={Extensions entre séries principales p-adiques et modulo p de G(F)},
	Author={Julien Hauseux},
	Subject={22E50},
	Keywords={extensions, séries principales, parties ordinaires, filtration de Bruhat}
}}

\title{Extensions entre séries principales $p$-adiques et modulo $p$ de $G(F)$}
\author{Julien Hauseux}
\date{}

\begin{document}

\maketitle



\tableofcontents

\section{Introduction}

\subsection*{Contexte}

Soient $F$ une extension finie de $\Qp$ et $G$ un groupe réductif connexe déployé sur $F$. On fixe une extension finie $E$ de $\Qp$.

Supposons $F=\Qp$, le centre de $G$ connexe et le groupe dérivé de $G$ simplement connexe (par exemple $G=\GL_n$ ou $G=\GSp_{2n}$). Dans \cite{BH}, Breuil et Herzig associent à toute représentation continue ordinaire générique $\rho : \Gal(\Qpbar/\Qp) \to \widehat{G}(E)$ où $\widehat{G}$ désigne le groupe dual de $G$, une représentation continue unitaire $\Pi^\mathrm{ord}(\rho)$ de $G(\Qp)$ sur un $E$-espace de Banach dont les constituants sont des séries principales. Ils conjecturent l'unicité des facteurs directs indécomposables de cette représentation étant donnés les gradués de leurs filtrations par le socle (\cite[Conjecture 3.5.1]{BH}). Cela conduit naturellement au problème de connaître les extensions entre séries principales $p$-adiques de $G(\Qp)$ dans le cas générique. Toutes ces questions ont également une variante modulo $p$.

Supposons seulement $F \neq \Qp$. Dans \cite{BP}, Breuil et Pa{\v{s}}k{\=u}nas ont montré que pour  $\GL_2(F)$ il n'existe aucune extension non scindée entre séries principales modulo $p$ distinctes. Cela laisse supposer un résultat similaire pour $G(F)$.

Dans cet article, nous utilisons le $\delta$-foncteur $\HOrd$ des parties ordinaires dérivées d'Emerton (\cite{Em2}) relatif à un sous-groupe de Borel pour avancer dans ces deux directions.

\subsection*{Principaux résultats}

Soient $B \subset G$ un sous-groupe de Borel et $T \subset B$ un tore maximal déployé. On note $B^- \subset G$ le sous-groupe de Borel opposé à $B$ par rapport à $T$. On note $W$ le groupe de Weyl de $(G,T)$, $\Delta$ les racines simples de $(G,B,T)$ et pour tout $\alpha \in \Delta$, on note $s_\alpha \in W$ la réflexion simple correspondante. On note $\varepsilon : F^\times \to \Zp^\times$ le caractère cyclotomique $p$-adique et $\Oe$ l'anneau des entiers de $E$.
On calcule les $\Ext^1$ dans les catégories abéliennes de représentations continues unitaires sur $E$ en utilisant les extensions de Yoneda.

Soient $\chi,\chi' : T(F) \to \Oe^\times \subset E^\times$ des caractères continus unitaires.
Le foncteur $\Ind$ est exact et admet le foncteur $\Ord$ comme quasi-inverse à gauche, donc il induit une injection $E$-linéaire
\begin{equation*}
\Ext^1_{T(F)}(\chi',\chi) \hookrightarrow \Ext^1_{G(F)} \left( \Ind \chi',\Ind \chi \right).
\end{equation*}
Lorsque $\chi' \neq \chi$ il n'existe aucune extension non scindée de $\chi'$ par $\chi$, donc les éventuelles extensions entre leurs induites sont construites autrement.

Supposons $F=\Qp$, le centre de $G$ connexe et le groupe dérivé de $G$ simplement connexe. On note $\theta$ la somme des poids fondamentaux relatifs à $\Delta$. Dans ce cas pour tout $\alpha \in \Delta$ tel que $s_\alpha(\chi) \neq \chi$, il existe aussi une extension non scindée entre les induites de $\chi \epsth$ et $s_\alpha(\chi) \epsth$. On la construit par induction parabolique à partir de l'unique extension non scindée entre les séries principales de $\GL_2(\Qp)$ correspondantes. Le résultat suivant (Théorème \ref{theo:ext1}) montre que l'on obtient toutes les extensions entre séries principales dans le cas générique par ces deux procédés.

\begin{theo} \label{theo:1}
Soit $\chi : T(\Qp) \to  \Oe^\times \subset E^\times$ un caractère continu unitaire.
\begin{enumerate}[(i)]
\item Si $\chi' : T(\Qp) \to  \Oe^\times \subset E^\times$ est un autre caractère continu unitaire, alors
\begin{equation*}
\Ext^1_{G(\Qp)} \left( \IndQp \chi' \epsth,\IndQp \chi \epsth \right) \neq 0
\end{equation*}
si et seulement si $\chi'=\chi$ ou $\chi'=s_\alpha(\chi)$ avec $\alpha \in \Delta$.
\item Soit $\alpha \in \Delta$. Si $s_\alpha(\chi) \neq \chi$, alors
\begin{equation*}
\dim_E \Ext^1_{G(\Qp)} \left( \IndQp s_\alpha(\chi) \epsth,\IndQp \chi \epsth \right) = 1.
\end{equation*}
\item Si $\chi$ est faiblement générique, alors le foncteur $\IndQp$ induit un isomorphisme $E$-linéaire
\begin{multline*}
\Ext_{T(\Qp)}^1 \left( \chi \epsth, \chi \epsth \right) \\
\iso \Ext_{G(\Qp)}^1 \left( \IndQp  \chi \epsth,\IndQp \chi \epsth \right).
\end{multline*}
\end{enumerate}
\end{theo}

Supposons seulement $F \neq \Qp$. Dans ce cas, le résultat suivant (Théorème \ref{theo:ext1F}) montre que toute extension entre deux séries principales est induite à partir d'une extension entre leurs caractères.

\begin{theo} \label{theo:2}
Soit $\chi : T(F) \to \Oe^\times \subset E^\times$ un caractère continu unitaire.
\begin{enumerate}[(i)]
\item Si $\chi' : T(F) \to \Oe^\times \subset E^\times$ est un autre caractère continu unitaire, alors
\begin{equation*}
\Ext_{G(F)}^1 \left( \Ind \chi', \Ind \chi \right) \neq 0
\end{equation*}
si et seulement si $\chi'=\chi$.
\item Le foncteur $\Ind$ induit un isomorphisme $E$-linéaire
\begin{equation*}
\Ext_{T(F)}^1 \left( \chi,\chi \right) \iso \Ext_{G(F)}^1 \left( \Ind \chi, \Ind \chi \right).
\end{equation*}
\end{enumerate}
\end{theo}

Enfin, on a les résultats analogues modulo $p$ (c'est-à-dire dans les catégories de représentations lisses admissibles sur le corps résiduel $\ke$ de $\Oe$). Ils sont démontrés dans les preuves des théorèmes \ref{theo:ext1} et \ref{theo:ext1F}.

\subsection*{Méthodes utilisées}

On procède par réduction modulo $p^k$ et dévissage. En caractéristique positive, on suit la stratégie d'Emerton qui utilise une propriété du $\delta$-foncteur $\HOrd$ dérivée de la relation d'adjonction entre les foncteurs $\Ind$ et $\Ord$. Pour cela, on commence par calculer $\HOrd$ sur une induite.

On fixe une uniformisante $\pe$ de $\Oe$ et un entier $k \geq 1$. On note $A$ l'anneau $\A{k}$, $\omega : F^\times \to A^\times$ l'image de $\varepsilon$ dans $A^\times$ et $\ell : W \to \Nbb$ la longueur relative à $\Delta$. On a le résultat suivant (Corollaire \ref{coro:HnOrd}).

\begin{theo} \label{theo:3}
Soient $\chi : T(F) \to A^\times$ un caractère lisse et $n \in \Nbb$. On suppose $A=\ke$ ou $n \leq 1$. Alors on a un isomorphisme $T(F)$-équivariant
\begin{equation*}
\HOrd[n]\left(\Ind \chi \omth\right) \cong \bigoplus_{[F:\Qp] \cdot \ell(w)=n} w(\chi) \omth.
\end{equation*}
\end{theo}

Détaillons les différentes étapes de ce calcul. Les foncteurs $\HOrd$ sont construits à partir des $A$-modules de cohomologie d'un sous-groupe compact $N_0 \subset B(F)$ et d'une action (dite de Hecke) d'un sous-monoïde $T^+ \subset T(F)$ sur ces derniers. Soit $U$ une représentation lisse localement admissible de $T(F)$ sur $A$.
Dans la section \ref{sec:2}, nous définissons une filtration classique (dite de Bruhat) de $\Ind U$ par des sous-$B(F)$-représentations. Nous calculons son gradué et nous montrons qu'elle induit une filtration des $A$-modules $\Hc(N_0,\Ind U)$.
Dans la section \ref{sec:3}, nous montrons comment la cohomologie et l'action de Hecke se décomposent à travers un dévissage de $N_0$. Puis, nous utilisons des dévissages successifs de $N_0$ pour calculer sa cohomologie à valeurs dans le gradué de la filtration de Bruhat de $\Ind U$ ainsi que l'action de Hecke de $T^+$.
Dans la section \ref{sec:4}, nous utilisons les résultats précédents pour calculer $\HOrd[1](\Ind U)$ en toute généralité (Corollaire \ref{coro:H1Ord}) et nous démontrons le théorème \ref{theo:3}.

Nous déterminons ensuite les extensions entre séries principales continues unitaires $p$-adiques de $G(F)$. Soient $\chi,\chi' : T(F) \to \Oe^\times \subset E^\times$ des caractères continus unitaires.
Dans la section \ref{sec:5}, nous utilisons le théorème \ref{theo:3} et la stratégie d'Emerton pour calculer les extensions entre les induites des réductions modulo $\pe^k$ de ces caractères. En particulier avec $k=1$, nous obtenons les analogues modulo $p$ des théorèmes \ref{theo:1} et \ref{theo:2}. Puis, par un passage à la limite projective et un produit tensoriel, nous en déduisons les théorèmes \ref{theo:1} et \ref{theo:2}. Enfin en supposant vraie une conjecture d'Emerton, nous calculons les extensions de longueur supérieure entre séries principales lisses modulo $p$ de $G(F)$ dans la catégorie des représentations lisses localement admissibles de $G(F)$ sur $\ke$ (Théorème \ref{theo:extn}).

\subsection*{Notations et conventions}

Soit $F$ une extension finie de $\Qp$.
On note $\val : F^\times \to \Q$ la valuation $p$-adique normalisée par $\val(p)=1$ et $|~|_p$ la valeur absolue $p$-adique correspondante.
On note $\varepsilon : F^\times \to \Zp^\times$ le caractère cyclotomique $p$-adique et $\omega : F^\times \to \Fp^\times$ sa réduction modulo $p$.
On note enfin $\Nrm : F^\times \to \Qp^\times$ l'application norme.
On a donc $\varepsilon(x) = \Nrm(x) |\Nrm(x)|_p$ pour tout $x \in F^\times$.

Soient $G$ un groupe réductif connexe déployé sur $F$, $B \subset G$ un sous-groupe de Borel et $T \subset B$ un tore maximal déployé. On note $B^- \subset G$ le sous-groupe de Borel opposé à $B$ par rapport à $T$ et $N$ le radical unipotent de $B$. On note enfin $S$ le plus grand sous-tore déployé de $\Res T$.

On note $\Phi^+$ les racines positives de $(G,B,T)$ et $\Delta \subset \Phi^+$ les racines simples.
On note $W$ le groupe de Weyl de $(G,T)$ et pour tout $\alpha \in \Phi^+$ on note $s_\alpha \in W$ la réflexion correspondante. On appelle décomposition réduite de $w \in W$ toute écriture de $w$ en un produit de longueur minimal de réflexions $s_\alpha$ avec $\alpha \in \Delta$ et on note $\ell(w)$ cette longueur. On note $w_0$ l'élément de longueur maximale de $W$ et $d$ l'entier $\ell(w_0) = \dim N$. Pour tout $w \in W$, on fixe un représentant $\w \in G(F)$ de $w$ dans le normalisateur de $T(F)$.

On fixe une extension finie $E$ de $\Qp$. On note $\Oe$ l'anneau des entiers de $E$ et $\ke$ le corps résiduel de $\Oe$. On fixe une uniformisante $\pe$ de $\Oe$. On désigne par $A$ une $\Oe$-algèbre locale artinienne de corps résiduel $\ke$ (par exemple $A=\A{k}$ avec $k \geq 1$ entier) et on note encore $\omega : F^\times \to A^\times$ l'image de $\varepsilon$ dans $A^\times$.

Par un groupe de Lie $p$-adique, on entend une variété analytique sur $\Qp$ munie d'une structure de groupe pour laquelle le produit et le passage à l'inverse sont analytiques sur $\Qp$. En particulier, sa dimension est calculée sur $\Qp$.
Les groupes $G(F)$, $B(F)$, $T(F)$\dots{} sont munis de leur structure naturelle de groupe de Lie $p$-adique.

Soit $H$ un groupe de Lie $p$-adique. Une représentation $V$ de $H$ sur $A$ est dite \emph{lisse} si pour tout $v \in V$ le fixateur de $v$ dans $H$ est ouvert, \emph{admissible} si de plus $V^{H_0}$ est de type fini sur $A$ pour tout sous-groupe ouvert $H_0$ de $H$ et \emph{localement admissible} si pour tout $v \in V$ la sous-$H$-représentation de $V$ sur $A$ engendrée par $v$ est admissible. En caractéristique nulle, on considère des représentations continues unitaires admissibles de $H$ sur des $E$-espaces de Banach (voir l'appendice \ref{app:rep}).

\subsection*{Remerciements}

Ce travail a été réalisé sous la direction de Christophe Breuil. Je lui exprime mes remerciements les plus sincères pour m'avoir fait part de ses idées, ainsi que pour ses explications et ses remarques.

\numberwithin{theo}{subsection}

\section{Filtration de Bruhat} \label{sec:2}

Soit $U$ une représentation lisse de $T(F)$ sur $A$. Par inflation et induction nous obtenons une représentation lisse $\Ind U$ de $G(F)$ sur $A$ (voir \cite[§ 4.1]{Em1}). Nous construisons, à partir de la décomposition de Bruhat, une filtration de cette induite par des sous-$B(F)$-représentations. Puis, nous montrons que l'on obtient ainsi une filtration des $A$-modules de cohomologie d'un sous-groupe ouvert compact $N_0$ de $N(F)$ à valeurs dans $\Ind U$. Enfin, nous adaptons ces résultats aux induites paraboliques d'un caractère.

\subsection{Définition et calcul du gradué}

On définit une filtration de $\Ind U$ par des sous-$B(F)$-représentations et on calcule son gradué.

\subsubsection*{Décomposition de Bruhat et filtration}

On considère les doubles classes $B^- \w B$ avec $w \in W$ ; il faut donc adapter les résultats classiques concernant les doubles classes $B^- \w B^-$, en remarquant que $B^- \w B = B^- \w \w_0 B^- \w_0$ et en remplaçant $w$ par $w w_0$ notamment.
Pour tout $w \in W$, on définit la cellule correspondante en posant
\begin{equation*}
C(w) \dfn B^- (F) \w B(F).
\end{equation*}
La décomposition de Bruhat (voir \cite[Partie II, § 1.9]{Jan}) s'écrit
\begin{equation*}
G(F) = \coprod_{w \in W} C(w).
\end{equation*}
On définit l'ordre de Bruhat sur $W$ de la façon suivante : pour tous $w,w' \in W$, on a $w' \leq w$ si et seulement si il existe une décomposition réduite $s_1 \dots s_{\ell(w)}$ de $w$ et des entiers $1 \leq k_1 < \dots < k_{\ell(w')} \leq \ell(w)$ tels que $w'=s_{k_1} \dots s_{k_{\ell(w')}}$. D'après \cite[Théorème 3.13]{BTC}, pour tout $w \in W$ on a
\begin{equation*}
\overline{C(w)} = \coprod_{w' \geq w} C(w')
\end{equation*}
avec $\overline{C(w)}$ l'adhérence de $C(w)$ dans $G(F)$ pour la topologie $p$-adique (voir \cite[Corollaire 3.15]{BTC}). On a donc une unique cellule ouverte $C(1)=B^-(F)B(F)$ appelée la grosse cellule ; elle est contenue dans tout ouvert $\BB$-biinvariant (c'est-à-dire $B^-(F)$-invariant par translation à gauche et $B(F)$-invariant par translation à droite) de $G(F)$. De même, on a une unique cellule fermée $C(w_0) = B^-(F)\w_0$. Pour tout $w \in W$, on pose
\begin{equation*}
G_w \dfn \coprod_{w' \leq w} C(w')
\end{equation*}
qui est par définition le plus petit ouvert $\BB$-biinvariant de $G(F)$ contenant $C(w)$. Par construction pour tous $w,w' \in W$, on a $G_{w'} \subset G_w$ si et seulement si $w' \leq w$. Enfin pour tout $r \in \llbrack -1,d \rrbrack$, on pose
\begin{equation*}
G_r \dfn \bigcup_{\ell(w)=r} G_w = \coprod_{\ell(w) \leq r} C(w).
\end{equation*}
On définit ainsi une filtration $(G_r)_{r \in \llbrack -1,d \rrbrack}$ de $G(F)$ par des sous-ensembles ouverts $\BB$-biinvariants. On a $G_{-1}=\emptyset$, $G_d=G_{w_0} = G(F)$ et $G_0 = G_1$ ($0 \in \Nbb, 1 \in W$) est la grosse cellule de $G(F)$. Par construction pour tout $w \in W$, la cellule $C(w)$ est fermée dans $G_w$ et dans $G_{\ell(w)}$.

Pour tout sous-ensemble $\BB$-biinvariant $C$ de $G(F)$, on définit une sous-$B(F)$-représentation de $\Ind U$ en posant
\begin{equation*}
\left(\Ind U\right)(C) \dfn \left\{f \in \Ind U ~|~ \supp(f) \subset C\right\}
\end{equation*}
où $\supp$ désigne le support d'une fonction localement constante (c'est-à-dire le sous-ensemble ouvert et fermé des points en lesquels elle ne s'annule pas). En particulier pour tout $r \in \llbrack -1,d \rrbrack$, on pose
\begin{equation*}
I_r \dfn \left(\Ind U\right)(G_r).
\end{equation*}
On définit ainsi la \emph{filtration de Bruhat} $(I_r)_{r \in \llbrack -1,d \rrbrack}$ de $\Ind U$. On a $I_{-1}=0$, $I_d = \Ind U$ et $I_0$ correspond aux fonctions à support dans la grosse cellule de $G(F)$.

\subsubsection*{Calcul du gradué}

On commence par quelques rappels topologiques.
Un espace séparé est dit localement profini si tout point admet une base de voisinages constituée d'ouverts compacts.
Tout sous-espace localement fermé (c'est-à-dire intersection d'un ouvert et d'un fermé) d'un espace localement profini est encore localement profini. L'image d'un espace localement profini dans un espace séparé par une application ouverte est encore localement profini. Un espace topologique est dit profini s'il est compact et localement profini.

\begin{lemm} \label{lemm:paracpt}
Tout recouvrement ouvert d'un espace profini admet un raffinement fini constitué d'ouverts compacts disjoints.
\end{lemm}

\begin{proof}
Soit $X$ un espace topologique profini. Puisque la topologie de $X$ admet une base constituée d'ouverts compacts, tout recouvrement ouvert admet un raffinement constitué d'ouverts compacts que l'on peut supposer fini par compacité de $X$. Si $(O_i)_{i \in I}$ est un tel recouvrement, alors
\begin{equation*}
\left( \left( \bigcap_{j \in J} O_j \right) \bigcap \left(\bigcap_{i \in I - J} X - O_i \right) \right)_{J \subset I}
\end{equation*}
est un raffinement fini de $(O_i)_{i \in I}$ constitué d'ouverts compacts disjoints.
\end{proof}

Le groupe $G(F)$ est localement profini (comme tout groupe de Lie $p$-adique) et le quotient $\BG$ est compact donc profini.
On déduit du lemme \ref{lemm:paracpt} que le fibré principal $\pi : G(F) \twoheadrightarrow \BG$ est trivial (c'est-à-dire qu'il admet une section continue). En effet, d'après \cite[Partie II, § 1.10]{Jan} il admet des sections locales, donc on peut trouver un recouvrement de $\BG$ par des ouverts trivialisants. En choisissant un raffinement de ce recouvrement par des ouverts disjoints, on peut recoller ces sections locales en une section globale. On note $\sigma : \BG \hookrightarrow G(F)$ une telle section continue.

On calcule à présent le gradué de la filtration $(I_r)_{r \in \llbrack -1,d \rrbrack}$. Pour tout sous-ensemble $\BB$-biinvariant $C$ de $G(F)$ on pose
\begin{multline*}
\C_C \dfn \{\text{$f : C \to U$ localement constante à support compact} \\
\text{modulo $B^-(F)$ | $f(bg) = b \cdot f(g)$ pour tous $b \in B^-(F), g \in C$}\}
\end{multline*}
et on munit ce $A$-module de l'action lisse de $B(F)$ par translation à droite.

\begin{lemm} \label{lemm:index}
Soit $C$ un ouvert $\BB$-biinvariant de $G(F)$. La restriction des fonctions à $C$ induit un isomorphisme $B(F)$-équivariant
\begin{equation*}
\left(\Ind U\right)(C) \cong \C_C.
\end{equation*}
\end{lemm}

\begin{proof}
L'injectivité est claire puisque les éléments du membre de gauche sont à support dans $C$. Pour la surjectivité on utilise le fait que les éléments du membre de droite sont à support compact modulo $B^-(F)$. En les prolongeant sur $G(F)$ par $0$, les fonctions obtenues sont donc encore localement constantes.
\end{proof}

\begin{prop} \label{prop:index}
Soient $C$ un ouvert $\BB$-biinvariant de $G(F)$ et $D$ un fermé $\BB$-biinvariant de $C$. Alors on a une suite exacte courte de représentations lisses de $B(F)$ sur $A$
\begin{equation*}
0 \to \C_{C - D} \to \C_C \to \C_D \to 0
\end{equation*}
où le premier morphisme non trivial est le prolongement sur $C$ par $0$ et le second est la restriction à $D$.
\end{prop}

\begin{proof}
Il s'agit de \cite[Proposition 1.8]{BZ} que nous redémontrons.
L'injectivité du premier morphisme non trivial est évidente. On vérifie l'exactitude de cette suite au milieu. Soit $f \in \C_C$ dont la restriction à $D$ est nulle. Comme $f$ est localement constante, elle s'annule sur un voisinage ouvert de $D$. On en déduit que $f$ est à support dans $C - D$. Donc $f$ est le prolongement sur $C$ par $0$ d'un élément de $\C_{C - D}$. On examine enfin la surjectivité du dernier morphisme non trivial. Soit donc $f \in \C_D$.

On montre qu'il existe $n \in \Nbb$ et des sous-ensembles compacts disjoints $(D_k)_{k \in \llbrack 1,n \rrbrack}$ de $\sigma(\BD) \subset D$ tels que :
\begin{itemize}
\item $\supp(f) = \bigcup_{k \in \llbrack 1,n \rrbrack} B^-(F) D_k$ ;
\item $f$ est constante sur $D_k$ pour tout $k \in \llbrack 1,n \rrbrack$ ;
\item $\pi(D_k)$ est ouvert dans $\BD$ pour tout $k \in \llbrack 1,n \rrbrack$.
\end{itemize}
À partir d'un recouvrement de $D$ par des ouverts sur lesquels $f$ est constante, on obtient un recouvrement ouvert du compact $(\sigma \circ \pi)(\supp(f))$. D'après le lemme \ref{lemm:paracpt}, ce recouvrement admet un raffinement fini constitué d'ouverts compacts disjoints $(D_k)_{k \in \llbrack 1,n \rrbrack}$. Pour tout $k \in \llbrack 1,n \rrbrack$, $f$ est constante sur $D_k$ par construction et $\pi$ induit un homéomorphisme de $(\sigma \circ \pi)(\supp(f))$ sur $\pi(\supp(f))$ qui est ouvert dans $\BD$, donc $\pi(D_k)$ est ouvert dans $\BD$. Enfin $\supp(f) = B^-(F) (\sigma \circ \pi)(\supp(f)) = \bigcup_{k \in \llbrack 1,n \rrbrack} B^-(F) D_k$.

Pour tout $k \in \llbrack 1,n \rrbrack$, on note $u_k \in U-\{0\}$ l'unique valeur de $f$ sur $D_k$ et $B_k^-$ le stabilisateur de $u_k$ dans $B^-(F)$. On montre qu'il existe des sous-ensembles ouverts $(C_k)_{k \in \llbrack 1,n \rrbrack}$ de $C$ tels que :
\begin{itemize}
\item $C_k \cap D = B_k^- D_k$ pour tout $k \in \llbrack 1,n \rrbrack$ ;
\item $\forall k,k' \in \llbrack 1,n \rrbrack, \forall b \in B^-(F), (b C_k) \cap C_{k'} \neq \emptyset \Rightarrow k=k'$ et $b \in B_k^-$.
\end{itemize}
Soit $k \in \llbrack 1,n \rrbrack$. Comme $\pi(D_k)$ est ouvert dans $\BD$, il existe un ouvert $O_k$ de $\BC$ tel que $O_k \cap (\BD) = \pi(D_k)$. Par compacité de $\pi(D_k)$, on peut supposer $O_k$ compact (quitte à recouvrir $O_k$ par des ouverts compacts de $\BC$ et à extraire un recouvrement fini de $\pi(D_k)$). Enfin puisque les $\pi(D_k)$ sont disjoints, on peut supposer les $O_k$ disjoints (quitte à remplacer $O_k$ par $O_k-\bigcup_{k' \in \llbrack 1,n \rrbrack - \{k\}} (O_{k'} \cap O_k)$).
On note $C_k$ le sous-ensemble $B^-_k \sigma(O_k)$ de $C$. C'est un ouvert de $C$ car $\sigma$ induit une trivialisation globale du fibré principal $C \twoheadrightarrow \BC$, $B^-_k$ est un ouvert de la fibre $B^-(F)$ et $O_k$ est un ouvert de la base $\BC$. Par construction $\sigma(O_k) \cap D = D_k$, d'où $C_k \cap D = B_k^- D_k$. De plus si $(b C_k) \cap C_k \neq \emptyset$, alors $b \in B_k^-$. Enfin $\pi(C_k)=O_k$ par construction et ces derniers sont disjoints, d'où $(B^- C_k) \cap C_{k'} = \emptyset$ lorsque $k \neq k'$.

Pour tout $k \in \llbrack 1,n \rrbrack$, on prolonge $f$ sur $C_k$ par $u_k$, puis sur $B^-(F) C_k$ par la formule $f(bg)=b \cdot f(g)$ pour tous $b \in B^-(F), g \in C$. Ce prolongement est bien défini grâce aux conditions sur les ouverts $(C_k)_{k \in \llbrack 1,n \rrbrack}$. Enfin on prolonge $f$ sur $C$ par $0$. Par construction $f$ est localement constante sur le sous-ensemble $f^{-1}(U-\{0\})=B^-(F) (\bigcup_{k \in \llbrack 1,n \rrbrack} C_k)$. Or $\pi(B^-(F) (\bigcup_{k \in \llbrack 1,n \rrbrack} C_k)) = \bigcup_{k \in \llbrack 1,n \rrbrack} O_k$ est compact. On en déduit que $f$ est localement constante et à support compact modulo $B^-(F)$. Ainsi $f \in \C_C$.
\end{proof}

Pour tout $r \in \llbrack 0,d \rrbrack$, on applique le lemme \ref{lemm:index} et la proposition \ref{prop:index} avec $C=G_r$ et $D=G_r - G_{r-1}$ et on obtient des isomorphismes $B(F)$-équivariants
\begin{equation*}
I_r / I_{r-1} \cong \C_{G_r} / \C_{G_{r-1}} \cong \C_{\coprod_{\ell(w)=r} C(w)} \cong \bigoplus_{\ell(w)=r} \C_{C(w)},
\end{equation*}
le dernier étant la conséquence du fait que $\coprod_{\ell(w)=r} C(w)$ est une partition finie de $G_r-G_{r-1}$ en sous-ensembles $\BB$-biinvariants fermés donc ouverts (par finitude).

Soit $w \in W$. Afin de donner une description plus explicite de la représentation $\C_{C(w)}$, on définit un sous-groupe fermé de $N$ stable sous l'action par conjugaison de $T$ en posant
\begin{equation} \label{Nw}
N_w \dfn N \cap (\w^{-1} N\w).
\end{equation}
On remarque que $N_w \cap N_{w_0w} = 1$ et on déduit de l'isomorphisme \eqref{prodNa} que $N_w N_{w_0w}=N$ (mais $N$ n'est pas en général le produit semi-direct de ces deux sous-groupes). On en conclut que le produit induit un homéomorphisme
\begin{equation*}
B^-(F) \times \{\w\} \times N_w(F) \iso C(w).
\end{equation*}
Ainsi, en notant $\Clis_c(N_w(F),U)$ le $A$-module constitué des fonctions localement constantes à support compact de $N_w(F)$ dans $U$, on a un isomorphisme $A$-linéaire
\begin{equation} \label{isoClis}
\C_{C(w)} \cong \Clis_c(N_w(F),U).
\end{equation}
On transporte l'action de $B(F)$ à travers cet isomorphisme. On décrit cette action sur $f \in \Clis_c(N_w(F),U)$ dans les cas suivants :
\begin{itemize}
\item si $t \in T(F)$, alors $(t f)(n)=(\w t \w^{-1}) \cdot f(t^{-1}nt)$ pour tout $n \in N_w(F)$ ;
\item si $n' \in N_w(F)$, alors $(n' f)(n)=f(nn')$ pour tout $n \in N_w(F)$ ;
\item si $n' \in N_{w_0w}(F)$ normalise $N_w(F)$, alors $(n' f)(n)=f(n'^{-1}nn')$ pour tout $n \in N_w(F)$.
\end{itemize}

\subsubsection*{Conclusion}

On a défini la filtration de Bruhat $(I_r)_{r \in \llbrack -1,d \rrbrack}$ de $\Ind U$ et pour tout $r \in \llbrack 0,d \rrbrack$, on a une suite exacte courte de représentations lisses de $B(F)$ sur $A$
\begin{equation} \label{filind}
0 \to I_{r-1} \to I_r \to \bigoplus_{\ell(w)=r} \Clis_c(N_w(F),U) \to 0.
\end{equation}

\subsection{Filtration de la cohomologie}

On montre que la filtration de Bruhat induit une filtration des $A$-modules de cohomologie d'un sous-groupe ouvert compact $N_0$ de $N(F)$ à valeurs dans $\Ind U$.

\subsubsection*{Rappels sur la cohomologie}

On fait quelques rappels sur la cohomologie d'un groupe de Lie $p$-adique en suivant \cite[§ 2.2]{Em2}. Soient $H$ un groupe de Lie $p$-adique et $V$ une représentation lisse de $H$ sur $A$. On note $\Hc(H,V)$ les $A$-modules de cohomologie de $H$ à valeurs dans $V$ calculés en utilisant des cochaînes localement constantes. On obtient ainsi un $\delta$-foncteur $\Hc(H,-)$ de la catégorie des représentations lisses de $H$ sur $A$ dans la catégorie des $A$-modules. On dit que $V$ est $H$-acyclique si $\Hc[n](H,V)=0$ pour tout entier $n>0$.

Soit $H_0$ un groupe de Lie $p$-adique compact. Dans ce cas on va montrer que le $\delta$-foncteur $\Hc(H_0,-)$ est universel. Soit $V$ une représentation lisse de $H_0$ sur $A$. On note $\Clis(H_0,V)$ le $A$-module constitué des fonctions localement constantes de $H_0$ dans $V$ muni de l'action de $H_0$ par translation à droite. En utilisant l'injection $H_0$-équivariante $V \hookrightarrow \Clis(H_0,V)$ qui envoie un élément de $V$ sur son orbite et le lemme \ref{lemm:inj} ci-dessous, on voit que les foncteurs $\Hc[n](H_0,-)$ sont effaçables pour tout entier $n>0$. En particulier, ils coïncident avec les foncteurs dérivés à droite du foncteur des $H_0$-invariants dans la catégorie des représentations lisses de $H_0$ sur $A$.

\begin{lemm} \label{lemm:inj}
Soient $H_0$ un groupe de Lie $p$-adique compact et $V$ un $A$-module. Alors le $A$-module $\Clis(H_0,V)$ muni de l'action de $H_0$ par translation à droite est une représentation lisse $H_0$-acyclique.
\end{lemm}

\begin{proof}
C'est un résultat classique. Par compacité $H_0$ est profini et $\Clis(H_0,V)$ est un $H_0$-module discret induit. Le lemme de Shapiro permet de conclure (voir \cite[Chapitre I, § 2.5]{Ser}).
\end{proof}

\subsubsection*{Calculs topologiques}

Soit $N_0$ un sous-groupe ouvert compact de $N(F)$.

\begin{lemm} \label{lemm:profini}
L'espace topologique $\BG/N_0$ est profini.
\end{lemm}

\begin{proof}
L'espace topologique $\BG$ est profini donc compact et localement profini. Comme le groupe $N_0$ est compact, le quotient $\BG/N_0$ est encore compact d'après \cite[Chapitre III, § 4, Proposition 2, Corollaire 1]{BbkTG14} et il est encore localement profini par passage au quotient dans un espace séparé. Ainsi $\BG/N_0$ est profini.
\end{proof}

\begin{prop} \label{prop:cohom}
Pour tout $r \in \llbrack 0,d \rrbrack$, l'inclusion $I_{r-1} \subset I_r$ induit des injections $\Hc(N_0,I_{r-1}) \hookrightarrow \Hc(N_0,I_r)$.
\end{prop}

\begin{proof}
Soit $r \in \llbrack 0,d \rrbrack$. Le foncteur des $N_0$-invariants étant exact à gauche, la proposition est trivialement vraie en degré $0$. Soit donc $n\geq1$ entier. L'application $\Hc[n](N_0,I_{r-1}) \to \Hc[n](N_0,I_r)$ est induite par la composition des cochaînes avec l'inclusion $\iota : I_{r-1} \hookrightarrow I_r$. Soit $\phi : N_0^n \to I_{r-1}$ une cochaîne. Supposons que la composée $\iota \circ \phi : N_0^n \to I_r$ soit un cobord, c'est-à-dire qu'il existe une cochaîne $\psi : N_0^{n-1} \to I_r$ telle que $\iota \circ \phi = \diff \psi$. Comme $\im(\phi) \subset I_{r-1}$ on a $\supp(f) \subset G_{r-1}$ pour tout $f \in \im (\phi)$, d'où
\begin{equation*}
D_1 \dfn \bigcup_{f \in \im(\phi)} \supp(f) \subset G_{r-1}.
\end{equation*}
Comme $\phi$ est localement constante sur $N_0^n$ qui est compact, cette réunion est finie. On en déduit que $D_1$ est fermé dans $G(F)$. On note
\begin{equation*}
D_2 \dfn G(F) - G_{r-1}
\end{equation*}
qui est également fermé dans $G(F)$.

Comme $D_1$ et $D_2$ sont $B^-(F)$-invariants par translation à gauche (le support d'un élément de $\Ind U$ l'étant) ce sont des fermés saturés de $G(F)$ par rapport au quotient $\BG$, donc leurs images dans ce dernier sont fermées. De plus l'application quotient $\BG \twoheadrightarrow \BG/N_0$ est fermée par compacité. Ainsi les images respectives $\overline{D}_1$ et $\overline{D}_2$ de $D_1$ et $D_2$ dans le quotient $\BG/N_0$ sont fermées.
Comme $G_{r-1}$ est $(B^-(F),N_0)$-biinvariant, $\overline{D}_1$ et $\overline{D}_2$ sont disjoints (l'un est dans l'image de $G_{r-1}$, l'autre dans l'image de son complémentaire qui est ici le complémentaire de son image). On a donc deux fermés disjoints $\overline{D}_1$ et $\overline{D}_2$ de $\BG/N_0$.

L'espace $\BG/N_0$ étant compact, il est normal (axiome de séparation T4). Il existe donc deux ouverts disjoints contenant $\overline{D}_1$ et $\overline{D}_2$ respectivement. Avec le complémentaire de $\overline{D}_1 \cup \overline{D}_2$, ils forment un recouvrement ouvert de $\BG/N_0$. Ce recouvrement admet un raffinement fini constitué d'ouverts disjoints d'après le lemme \ref{lemm:paracpt}. Soient $\overline{C}_1$ la réunion de ceux qui intersectent $\overline{D}_1$ et $\overline{C}_2$ la réunion de ceux dont l'intersection avec $\overline{D}_1$ est vide. Alors $\overline{D}_1 \subset \overline{C}_1$, $\overline{D}_2 \subset \overline{C}_2$ et $\BG/N_0 = \overline{C}_1 \amalg \overline{C}_2$. En prenant les images réciproques de $\overline{C}_1$ et $\overline{C}_2$ par l'application quotient $G \twoheadrightarrow \BG/N_0$, on trouve une partition de $G(F)$ en deux ouverts $(B^-(F),N_0)$-biinvariants $C_1$ et $C_2$ contenant $D_1$ et $D_2$ respectivement. En particulier $C_1$ est un ouvert fermé de $G(F)$ inclus dans $G_{r-1}$.

La fonction caractéristique $1_{C_1} : G(F) \to A$ de $C_1$ est donc localement constante, $B^-(F)$-invariante par translation à gauche et à support dans $G_{r-1}$. On en déduit que pour tout $f \in I_r$, la fonction $1_{C_1} f$ est bien un élément de $I_{r-1}$ et on définit ainsi une application $A$-linéaire $\varrho : I_r \to I_{r-1}$. De plus $1_{C_1}$ est $N_0$-invariante par translation à droite, donc l'application $\varrho$ est aussi $N_0$-équivariante. On en déduit que $\diff(\varrho \circ \psi) = \varrho \circ (\diff \psi) = \varrho \circ \iota \circ \phi = \phi$, la dernière égalité résultant du fait que $1_{C_1}$ vaut $1$ sur le support de tout élément de $\im(\phi)$. Ainsi $\phi$ est un cobord (celui de la cochaîne $\varrho \circ \psi : N_0^{n-1} \to I_{r-1}$).
\end{proof}

Soit $r \in \llbrack 0,d \rrbrack$. La proposition \ref{prop:cohom} montre que les morphismes connectants de la longue suite exacte de cohomologie obtenue à partir de la suite exacte courte \eqref{filind} sont tous nuls.

\subsubsection*{Conclusion}

La filtration de Bruhat induit une filtration $(\Hc(N_0,I_r))_{r \in \llbrack -1,d \rrbrack}$ des $A$-modules $\Hc(N_0,\Ind U)$ et pour tout $r \in \llbrack 0,d \rrbrack$, on a des suites exactes courtes de $A$-modules
\begin{equation} \label{filcohom}
0 \to \Hc(N_0,I_{r-1}) \to \Hc(N_0,I_r) \to \bigoplus_{\ell(w)=r} \Hc(N_0,\Clis_c(N_w(F),U)) \to 0.
\end{equation}

\subsection{Variante pour les induites paraboliques} \label{sec:para}

Soient $P \subset G$ un sous-groupe parabolique standard (c'est-à-dire contenant $B$) et $L \subset P$ le sous-groupe de Levi standard (c'est-à-dire contenant $T$). On note $P^- \subset G$ le sous-groupe parabolique standard opposé à $P$ par rapport à $L$. Il est caractérisé par l'égalité $L = P \cap P^-$ et $L \subset P^-$ est le sous-groupe de Levi standard. On note $W_L \subset W$ le groupe de Weyl de $(L,T)$ et $w_{L,0} \in W_L$ l'élément de longueur maximale. On a
\begin{equation*}
P=B\dot{W}_LB \quad \text{et} \quad P^- = B^-\dot{W}_LB^-.
\end{equation*}
On adapte les résultats des sous-sections précédentes aux induites paraboliques.

\subsubsection*{Décomposition de Bruhat généralisée}

On considère les doubles classes $P^- \w B$ avec $w \in W$ ; il faut encore adapter les résultats classiques (qui concernent les doubles classes $P^- \w B^-$).
Soit $w \in W$. D'après \cite[Proposition 3.9]{BTC}, il existe un unique élément $w_{\min}$ de longueur minimale dans $W_Lw$ et il est caractérisé par le fait que $\ell(w_Lw_{\min}) = \ell(w_{\min}) + \ell(w_L)$ pour tout $w_L \in W_L$. On en déduit que $w_{L,0} w_{\min}$ est l'unique élément de longueur maximale dans $W_Lw$ et il est caractérisé par le fait que $\ell(w_L\wP) = \ell(\wP) - \ell(w_L)$ pour tout $w_L \in W_L$. On définit
\begin{equation*}
\WP \dfn \{\text{$w \in W$ | $w$ de longueur maximale dans $W_Lw$}\}.
\end{equation*}

\begin{lemm} \label{lemm:P1}
Soit $\wP \in \WP$. On a la décomposition
\begin{equation*}
P^-(F) \wPd B(F) = \coprod_{w_L \in W_L} B^-(F) \w_L \wPd B(F).
\end{equation*}
\end{lemm}

\begin{proof}
Par définition de $P^-$, on a
\begin{equation*}
P^-(F)\wPd B(F) = \left( \coprod_{w_L \in W_L} B^-(F) \w_L B^-(F) \right) \left(B^-(F)\wPd \w_0B^-(F) \right) \w_0.
\end{equation*}
Pour tout $w_L \in W_L$, on a
\begin{align*}
\ell(w_L \wP w_0) &= \ell(w_0) - \ell(w_L\wP) \\
&= \ell(w_0) - \ell(\wP) + \ell(w_L) \\
&= \ell(w_L) + \ell(\wP w_0).
\end{align*}
Donc d'après \cite[Lemme 3.4]{BTC}, on a
\begin{equation*}
\left( B^-(F) \w_L B^-(F) \right) \left( B^-(F) \wPd \w_0 B^-(F) \right) = B^-(F) \w_L \wPd \w_0 B^-(F)
\end{equation*}
et on en déduit le lemme en translatant à droite par $\w_0$.
\end{proof}

\begin{lemm} \label{lemm:P2}
Soit $\wP \in \WP$. La projection $\BG \twoheadrightarrow P^-(F) \backslash G(F)$ induit un homéomorphisme
\begin{equation*}
B^-(F) \backslash B^-(F) \wPd B(F) \iso P^-(F) \backslash P^-(F) \wPd B(F).
\end{equation*}
\end{lemm}

\begin{proof}
La preuve du lemme précédent montre que $\wP w_0$ est de longueur minimale dans $W_L \wP w_0$, donc d'après \cite[Proposition 3.16]{BTC} la projection $\BG \twoheadrightarrow P^-(F) \backslash G(F)$ induit un homéomorphisme
\begin{equation*}
B^-(F) \backslash B^-(F) \wPd \w_0 B^-(F) \iso P^-(F) \backslash P^-(F) \wPd \w_0 B^-(F)
\end{equation*}
et on en déduit le lemme en translatant à droite par $\w_0$.
\end{proof}

Pour tout $\wP \in \WP$, on définit la cellule correspondante en posant
\begin{equation*}
C_P(\wP) \dfn P^-(F) \wPd B(F)
\end{equation*}
et en utilisant le lemme \ref{lemm:P1} on voit que l'on a
\begin{equation*}
\overline{C_P(\wP)} = \coprod_{\wP' \geq \wP} C_P(\wP').
\end{equation*}
Pour tout $r \in \llbrack -1,d \rrbrack$, on pose
\begin{equation*}
G_r^P \dfn \bigcup_{\ell(\wP)=r} G_{\wP} = \coprod_{\ell(\wP) \leq r} C_P(\wP),
\end{equation*}
l'égalité résultant de la décomposition de Bruhat et du lemme \ref{lemm:P1}.

\subsubsection*{Filtration et gradué}

Soit $U$ une représentation lisse de $L(F)$ sur $A$. Par inflation et induction parabolique, on obtient une représentations lisse $\IndP U$ de $G(F)$ sur $A$ (voir \cite[§ 4.1]{Em1}). Pour tout sous-ensemble $(P^-(F),B(F))$-biinvariant $C$ de $G(F)$, on définit une sous-$B(F)$-représentation de $\IndP U$ en posant
\begin{equation*}
\left(\IndP U\right)(C) \dfn \left\{f \in \IndP U ~|~ \supp(f) \subset C\right\} \\
\end{equation*}
et on définit une représentation lisse de $B(F)$ sur $A$ en faisant agir $B(F)$ par translation à droite sur le $A$-module
\begin{multline*}
\C^P_C \dfn \{\text{$f : C \to U$ localement constante à support compact} \\
\text{modulo $P^-(F)$ | $f(pg) = p \cdot f(g)$ pour tous $p \in P^-(F), g \in C$}\}.
\end{multline*}
Enfin pour tout $r \in \llbrack -1,d \rrbrack$, on pose
\begin{equation*}
I_r^P \dfn \left(\IndP U\right)\left(G_r^P\right).
\end{equation*}

Avec ces définitions, le lemme \ref{lemm:index} et la proposition \ref{prop:index} ainsi que leurs démonstrations sont vrais tels quels avec $P^-$ et $\C^P$ au lieu de $B^-$ et $\C$. Pour tout $r \in \llbrack 0,d \rrbrack$, on obtient donc une suite exacte courte de représentations lisses de $B(F)$ sur $A$
\begin{equation*}
0 \to I^P_{r-1} \to I^P_r \to \bigoplus_{\ell(\wP)=r} \C^P_{C_P(\wP)} \to 0.
\end{equation*}
De même, le lemme \ref{lemm:profini} et la proposition \ref{prop:cohom} ainsi que leurs démonstrations sont vrais tels quels avec $P^-$ et $I^P$ au lieu de $B^-$ et $I$. Pour tout $r \in \llbrack 0,d \rrbrack$, on obtient donc des suites exactes courtes de $A$-modules
\begin{equation*}
0 \to \Hc(N_0,I^P_{r-1}) \to \Hc(N_0,I^P_r) \to \bigoplus_{\ell(\wP)=r} \Hc(N_0,\C^P_{C_P(\wP)}) \to 0.
\end{equation*}

On fixe un caractère algébrique de $G$ que l'on note\footnote{Cette notation est justifiée par le fait qu'un tel caractère est toujours le déterminant d'une représentation algébrique fidèle de $G$.} $\detfr$.
On se restreint au cas où $U = \eta \circ \detfr$ avec $\eta : F^\times \to A^\times$ un caractère lisse. On compare les gradués des filtrations de $\IndP \eta \circ \detfr$ et $\Ind \eta \circ \detfr$.

\begin{lemm}
Soit $\wP \in \WP$. On a un isomorphisme $B(F)$-équivariant
\begin{equation*}
\C^P_{C_P(\wP)} \cong \C_{C(\wP)}.
\end{equation*}
\end{lemm}

\begin{proof}
On vérifie que la restriction à $B^-(F) \wPd B(F)$ définit bien une application $B(F)$-équivariante
\begin{equation*}
\varrho : \C^P_{C_P(\wP)} \to \C_{C(\wP)}.
\end{equation*}
Soit $f \in \C^P_{C_P(\wP)}$. Alors $\varrho(f)$ est localement constante et à support compact modulo $B^-(F)$ d'après le lemme \ref{lemm:P2}. Comme $\detfr$ est un caractère algébrique, il est trivial sur tout sous-groupe unipotent de $G$. Ainsi pour tout $b \in B^-(F) \subset P^-(F)$, on a $\detfr b=\detfr t = \detfr l$ avec $t$ et $l$ les images de $b$ dans $T(F)$ et $L(F)$ respectivement. On en déduit que $\varrho(f)$ vérifie bien $\varrho(f)(bg) = \eta(\detfr b) \varrho(f)(g)$ pour tous $b \in B^-(F), g \in C(\wP)$, d'où $\varrho(f) \in \C_{C(\wP)}$.
L'application $\varrho$ ainsi définie est visiblement $B(F)$-équivariante et en utilisant encore le lemme \ref{lemm:P2} on voit qu'elle est bijective.
\end{proof}

On peut donc utiliser l'isomorphisme \eqref{isoClis} pour expliciter le gradué de la filtration de $\IndP \eta \circ \detfr$.

\subsubsection*{Conclusion}

On a défini une filtration $(I^P_r)_{r \in \llbrack -1,d \rrbrack}$ de $\IndP \eta \circ \detfr$ et pour tout $r \in \llbrack 0,d \rrbrack$, on a une suite exacte courte de représentations lisses de $B(F)$ sur $A$
\begin{equation} \label{filindP}
0 \to I^P_{r-1} \to I^P_r \to \bigoplus_{\ell(\wP)=r} \Clis_c(N_{\wP}(F),\eta \circ \detfr) \to 0
\end{equation}
qui induit des suites exactes courtes de $A$-modules en cohomologie
\begin{equation} \label{filcohomP}
0 \to \Hc(N_0,I^P_{r-1}) \to \Hc(N_0,I^P_r) \to \bigoplus_{\ell(\wP)=r} \Hc(N_0,\Clis_c(N_{\wP}(F),\eta \circ \detfr)) \to 0.
\end{equation}

\section{Cohomologie et action de Hecke} \label{sec:3}

Soit $N_0$ un sous-groupe ouvert compact standard de $N(F)$ compatible avec la décomposition radicielle (voir l'appendice \ref{app:Lie}). Nous définissons une action d'un sous-monoïde $T^+$ de $T(F)$ sur les $A$-modules de cohomologie de certains sous-groupes fermés $\Nz$ de $N_0$. Puis, nous étudions ces $A$-modules et cette action à travers un dévissage de $\Nz$. Enfin, nous utilisons des dévissages successifs de $N_0$ pour calculer sa cohomologie à valeurs dans le gradué de la filtration de Bruhat d'une représentation induite de $G(F)$ sur $A$, ainsi que l'action de $T^+$ sur ces $A$-modules.

\subsection{Définition et premières propriétés}

On définit un sous-monoïde $T^+ \subset T(F)$ et une action de $T^+$ sur les $A$-modules de cohomologie de certains sous-groupes compacts $\Nz \subset N_0$. On vérifie que ces définitions généralisent celles de \cite{Em1,Em2} et on calcule la cohomologie de $\Nz$ et l'action de $T^+$ en degré maximal.

\subsubsection*{Préliminaires}

Soient $H$ un groupe de Lie $p$-adique et $H^+$ un sous-monoïde de $H$. On suppose que $H^+$ contient un sous-groupe ouvert $H_0$ de $H$. On dit qu'une représentation $V$ de $H^+$ sur $A$ est lisse si l'action de $H_0$ sur $V$ est lisse.

\begin{lemm} \label{lemm:mnd}
La sous-catégorie pleine des représentations lisses de $H^+$ sur $A$ est abélienne. Elle possède suffisamment d'injectifs et ces derniers sont encore injectifs dans la catégorie des représentations lisses de $H_0$ sur $A$.
\end{lemm}

\begin{proof}
C'est une adaptation immédiate de \cite[Lemme 2.2.6]{Em1}, \cite[Proposition 2.2.1]{Em2} et \cite[Proposition 2.2.2]{Em2}.
\end{proof}

Soit $T_0 = \ker \nu$ avec $\nu$ défini par \eqref{nu} ci-dessous. C'est un sous-groupe ouvert de $T(F)$ (car la valuation de $F$ est discrète). On définit un sous-monoïde de $T(F)$ en posant
\begin{equation*}
T^+ \dfn \{t \in T(F) ~|~ \forall \alpha \in \Delta, \val (\alpha(t)) \geq 0\}.
\end{equation*}
On a donc $T_0 \subset T^+$.
On vérifie que cette définition de $T^+$ est identique à celle de \cite{Em1,Em2} et on montre l'existence de certains éléments particuliers de $T^+$.

\begin{lemm} \label{lemm:Tp}
Soit $t \in T(F)$. Alors $t \in T^+$ si et seulement si $t N_0 t^{-1} \subset N_0$.
\end{lemm}

\begin{proof}
Puisque $N_0$ est un sous-groupe standard de $N(F)$ compatible avec la décomposition radicielle, on a $\Lie(N_0) = \bigoplus_{\alpha \in \Phi^+} M_{\alpha,0}$ avec $M_{\alpha,0}$ un $\Zp$-réseau borné du sous-espace propre $\Lie(N_\alpha)$. Ainsi l'action adjointe de $t \in T(F)$ sur $\Lie(N)$ stabilise $M_{\alpha,0}$ si et seulement si $(\val \circ \alpha)(t) \geq 0$. On en déduit le lemme en utilisant l'exponentielle.
\end{proof}

\begin{lemm} \label{lemm:tau}
Il existe $\tau \in T^+$ vérifiant les propriétés équivalentes suivantes.
\begin{enumerate}[(i)]
\item $\forall \alpha \in \Delta, \val(\alpha(\tau))>0$.
\item $\left( \tau^k N_0 \tau^{-k} \right)_{k \in \Nbb}$ est une base de voisinages de $1 \in N_0$.
\item $N(F) = \bigcup_{k \in \Nbb} \tau^{-k} N_0 \tau^k$.
\item $\tau^{-1}$ et $T^+$ engendrent $T(F)$ comme monoïde.
\end{enumerate}
\end{lemm}

\begin{proof}
L'équivalence entre (i) et (iv) est immédiate en utilisant la définition de $T^+$. Pour l'équivalence entre (i), (ii) et (iii) on procède comme dans la preuve du lemme précédent. Pour prouver l'existence, on considère le morphisme de groupes continu
\begin{equation} \label{nu}
\nu : T(F) \to \Q^\Delta, t \mapsto ((\val \circ \alpha)(t))_{\alpha \in \Delta}.
\end{equation}
Pour tout $\beta \in \Delta$, on a $\nu(\beta^\vee(p)) = (\langle \alpha, \beta^\vee \rangle)_{\alpha \in \Delta}$. Comme l'accouplement $\langle ~,~ \rangle : \Z[\Delta] \times \Z[\Delta^\vee] \to \Z$ induit une dualité lorsque l'on tensorise par $\Q$ sur $\Z$, on en déduit que $\im \nu$ engendre $\Q^\Delta$ sur $\Q$.
Ainsi pour tout $c \in \Nbb$, il existe $t \in T(F)$ tel que $\val(\alpha(t)) \geq c$ pour tout $\alpha \in \Delta$. En particulier, il existe $\tau \in T(F)$ tel que $(\val \circ \alpha)(\tau)>0$ pour tout $\alpha \in \Phi^+$ (puisque $\Delta$ engendre $\Phi^+$ comme monoïde).
\end{proof}

On rappelle que $S$ est le plus grand sous-tore déployé de $\Res T$. On définit un sous-groupe ouvert et un sous-monoïde de $S(\Qp) \subset T(F)$ en posant
\begin{equation*}
S_0 \dfn S(\Qp) \cap T_0 \quad \text{et} \quad S^+ \dfn S(\Qp) \cap T^+.
\end{equation*}
On a donc $S_0 \subset S^+$. La remarque suivante met en évidence le lien entre les (co)caractères algébriques de $S$ et de $T$.

\begin{rema} \label{rema:S}
Le morphisme canonique $(\Res T)_F \to T$ induit un isomorphisme $S_F \iso T$. Comme $S$ est déployé, tout (co)caractère algébrique de $S_F$ est induit par un (co)caractère algébrique de $S$. On en conclut que les (co)caractères algébriques de $S$ s'identifient naturellement aux (co)caractères algébriques de $T$.
En particulier, la restriction à $S(\Qp)$ d'un morphisme $T(F) \to F^\times$ induit par un caractère algébrique de $T$ est à valeurs dans $\Qp^\times$ et la restriction à $\Qp^\times$ d'un morphisme $F^\times \to T(F)$ induit par un cocaractère algébrique de $T$ est à valeurs dans $S(\Qp)$.
\end{rema}

\subsubsection*{Action de Hecke}

Soit $\N \subset N$ un sous-groupe fermé stable sous l'action par conjugaison de $T$. On note $\Nz$ l'intersection de $\N(F)$ avec $N_0$. C'est un sous-groupe ouvert compact standard de $\N(F)$ stable sous l'action par conjugaison de $T^+$. On note $\NT$ le sous-monoïde de $T(F) \ltimes \N(F)$ engendré par $T^+$ et $\Nz$. Il contient le sous-groupe ouvert $T_0 \ltimes \Nz$.

Soit $V$ une représentation lisse de $\NT$ sur $A$. On définit l'\emph{action de Hecke} de $t \in T^+$ sur les $A$-modules $\Hc(\Nz,V)$ comme la composée
\begin{equation} \label{hecke}
\Hc(\Nz,V) \to \Hc(t\Nz t^{-1},V) \to \Hc(\Nz,V)
\end{equation}
où le premier morphisme est induit par l'action de $t$ sur $V$ et le second est la corestriction de $t\Nz t^{-1}$ à $\Nz$.

\begin{lemm}
Les $A$-modules $\Hc(\Nz,V)$ munis de l'action de Hecke sont des représentations lisses de $T^+$ sur $A$.
\end{lemm}

\begin{proof}
On montre d'abord le lemme en degré $0$. L'action de Hecke de $T^+$ sur $V^{\Nz}$ est une action de monoïde d'après \cite[Lemme 3.1.4]{Em1}. Comme le groupe $T_0 \subset T^+$ normalise $\Nz$, l'action de Hecke de $T_0$ sur $V^{\Nz}$ coïncide avec l'action lisse de $T_0$ sur $V^{\Nz} \subset V$. On en déduit que l'action de Hecke de $T^+$ sur $V^{\Nz}$ est lisse. En degré supérieur, on calcule les $A$-modules $\Hc(\Nz,V)$ à l'aide d'une résolution injective $V \hookrightarrow I^\bullet$ dans la catégorie des représentations lisses de $\NT$ sur $A$ (voir le lemme \ref{lemm:mnd} et \cite[Proposition 2.1.11]{Em2}) et on utilise le résultat en degré $0$ avec $(I^\bullet)^{\Nz}$.
\end{proof}

Par naturalité des morphismes de \eqref{hecke}, tout morphisme de représentations lisses de $\NT$ sur $A$ induit un morphisme $T^+$-équivariant pour l'action de Hecke en cohomologie.
Les foncteurs $\Hc(\Nz,-)$ forment donc un $\delta$-foncteur universel de la catégorie des représentations lisses de $\NT$ sur $A$ dans la catégorie des représentations lisses de $T^+$ sur $A$.

\begin{rema} \label{rema:heckeS}
On peut adapter ce qui précède à $S$ et $\Res N$.
Soient $^\backprime \N \subset \Res N$ un sous-groupe fermé stable sous l'action par conjugaison de $S$ et $^\backprime \Nz$ l'intersection de ${^\backprime \N}(\Qp)$ avec $N_0$. Si $V$ est une représentation lisse de $S^+ \ltimes {^\backprime \Nz}$ sur $A$, alors on définit l'action de Hecke de $s \in S^+$ sur les $A$-modules $\Hc(^\backprime \Nz,V)$ par \eqref{hecke} avec $s$ et $^\backprime \Nz$ au lieu de $t$ et $\Nz$.
On obtient ainsi un $\delta$-foncteur universel de la catégorie des représentations lisses de $S^+ \ltimes {^\backprime \Nz}$ sur $A$ dans la catégorie des représentations lisses de $S^+$ sur $A$.
Cette définition est compatible avec la précédente : si $^\backprime \N = \Res \N$, alors $^\backprime \Nz=\Nz$ et l'action de Hecke $S^+$ est la restriction de l'action de Hecke de $T^+$.
\end{rema}

\subsubsection*{Calcul en degré maximal}

On rappelle que $\dim \Nz$ désigne la dimension de $\Nz$ en tant que variété analytique sur $\Qp$. On a donc $\dim \Nz = [F:\Qp] \cdot \dim \N$ avec $\dim \N$ la dimension du groupe algébrique $\N$ sur $F$.

\begin{lemm} \label{lemm:Hnul}
Soit $n \in \Nbb$. Si $n> \dim \Nz$, alors $\Hc[n](\Nz,V)=0$.
\end{lemm}

\begin{proof}
Voir \cite[Lemme 3.5.4]{Em2}.
\end{proof}

On note $V_{\Nz}$ le $A$-module des coinvariants par $\Nz$, c'est-à-dire le quotient de $V$ par le sous-$A$-module engendré par les éléments de la forme $n \cdot v - v$ avec $n \in \Nz$ et $v \in V$. On note $\alpha$ le caractère algébrique de la représentation adjointe de $T$ sur $\det_F \Lie(\N)$.

\begin{prop} \label{prop:dimN}
Soit $n \in \Nbb$. Si $n=\dim \Nz$, alors on a un isomorphisme naturel $T^+$-équivariant
\begin{equation*}
\Hc[n](\Nz,V) \cong V_{\Nz} \otimes (\oma),
\end{equation*}
l'action de $t \in T^+$ sur $V_{\Nz}$ étant la composée
\begin{equation*}
V_{\Nz} \to V_{t\Nz t^{-1}} \to V_{\Nz}
\end{equation*}
où le premier morphisme est induit par l'action de $t$ sur $V$ et le second est la projection naturelle.
\end{prop}

\begin{proof}
Si $M$ est un $\Zp$-module libre de rang $r$, alors on note $M^\vee = \Hom_{\Zp}(M,\Zp)$ son dual et $\det_{\Zp} M = \Lambda_{\Zp}^r M$ sa plus haute puissance extérieure.
D'après \cite[Proposition 3.5.6]{Em2}, on a un isomorphisme $A$-linéaire naturel
\begin{equation} \label{dimN}
\Hc[n](\Nz,V) \iso V_{\Nz} \otimes_{\Zp} \det{_{\Zp}} \Lie(\Nz)^\vee.
\end{equation}
Par définition, $T(F)$ agit sur $\det_F \Lie(\N)$ à travers le caractère $\alpha$. On en déduit que l'action de $T(F)$ sur $\det_{\Qp} \Lie(\N)$ est donnée par le caractère $\Nrm \circ \alpha$.
Si $t \in T^+$, alors on a un diagramme commutatif
\begin{equation} \label{CD} \begin{tikzcd}[column sep=small,row sep=small]
\Hc[n](\Nz,V) \dar \rar{\sim} & V_{\Nz} \otimes_{\Zp} \det_{\Zp} \Lie(\Nz)^\vee \dar \\
\Hc[n](t\Nz t^{-1},V) \dar \rar{\sim} & V_{t\Nz t^{-1}} \otimes_{\Zp} (\det_{\Zp} \Lie(t\Nz t^{-1}))^\vee \dar \\
\Hc[n](\Nz,V) \rar{\sim} & V_{\Nz} \otimes_{\Zp} \det_{\Zp} \Lie(\Nz)^\vee
\end{tikzcd} \end{equation}
où les morphismes horizontaux correspondent à \eqref{dimN}, les morphismes verticaux de gauche sont ceux de la composée \eqref{hecke} et ceux de droite sont les suivants :
\begin{itemize}
\item le premier est le produit tensoriel du morphisme induit par l'action de $t$ sur $V$ avec la multiplication par $(\Nrm \circ \alpha)(t)^{-1}$ ;
\item le second est le produit tensoriel de la projection naturelle avec la multiplication par $|(\Nrm \circ \alpha)(t)|_p^{-1}$.
\end{itemize}
La commutativité du carré supérieur du diagramme \eqref{CD} est une conséquence de la naturalité de l'isomorphisme \eqref{dimN} tandis que celle du carré inférieur résulte de \cite[Lemme 3.5.10]{Em2}.
Enfin, la composée des morphismes verticaux de droite coïncide avec l'action de l'énoncé sur $V_{\Nz}$ tordue par le caractère $\oma$.
\end{proof}

\subsection{Étude à travers un dévissage}

On garde les notations précédentes : $\N, \Nz, V$. Le but de cette sous-section est de montrer comment la cohomologie de $\Nz$ à valeurs dans $V$ et l'action de Hecke de $T^+$ sur les $A$-modules $\Hc(\Nz,V)$ se décomposent à travers un dévissage de $\Nz$.

\subsubsection*{Cohomologie}

Soient $\N' \subset \N$ un sous-groupe fermé distingué stable sous l'action par conjugaison de $T$ et $\N''$ le groupe quotient $\N/\N'$. En utilisant l'isomorphisme \eqref{prodNa}, on voit que l'on a un produit semi-direct $\N = \N'' \ltimes \N'$. On note $\Nz'$ et $\Nz''$ les intersections respectives de $\N'(F)$ et $\N''(F)$ avec $N_0$. Ce sont des sous-groupes ouverts compacts standards de $\N'(F)$ et $\N''(F)$ respectivement. En utilisant la proposition \ref{prop:rad}, on voit que l'on a une suite exacte courte scindée de groupes topologiques
\begin{equation*}
1 \to \Nz' \to \Nz \to \Nz'' \to 1.
\end{equation*}

On veut exprimer la cohomologie de $\Nz$ à valeurs dans $V$ à travers ce dévissage. En degré $0$ on a l'égalité $V^{\Nz}=(V^{\Nz'})^{\Nz''}$ et en degré supérieur on a une suite spectrale de Hochschild-Serre (voir \cite[Chapitre I, § 2.6]{Ser}) de $A$-modules
\begin{equation} \label{HS}
\Hc[i](\Nz'',\Hc[j](\Nz',V)) \Rightarrow \Hc[i+j](\Nz,V)
\end{equation}
où pour tout $j \in \Nbb$, l'action de $n'' \in \Nz''$ sur $\Hc[j](\Nz',V)$ est induite au niveau des cochaînes par l'application $\phi \mapsto n''\phi$ définie par
\begin{equation*}
(n'' \phi)(n_1',\dots,n_j') = n'' \cdot \phi(n''^{-1}n_1'n'',\dots,n''^{-1}n_j'n'')
\end{equation*}
pour tout $(n_1',\dots,n_j') \in \Nz'^j$.

Si $\dim \Nz''=1$, alors la cohomologie de $\Nz''$ est nulle en degré strictement plus grand que $1$ d'après le lemme \ref{lemm:Hnul}. On déduit dans ce cas de la suite spectrale \eqref{HS} des suites exactes courtes de $A$-modules
\begin{equation} \label{HSSE}
0 \to \Hc[1](\Nz'',\Hc[n-1](\Nz',V)) \to \Hc[n](\Nz,V) \to \Hc[n](\Nz',V)^{\Nz''} \to 0.
\end{equation}
pour tout entier $n>0$. Le second morphisme non trivial est la restriction et lorsque $n=1$, le premier morphisme non trivial est l'inflation.

\subsubsection*{Action de Hecke}

L'action par conjugaison de $T^+$ sur $\Nz$ stabilise $\Nz'$ et $\Nz''$ et pour tout $t \in T^+$, on a une suite exacte courte scindée de groupes topologiques
\begin{equation*}
1 \to t\Nz't^{-1} \to t\Nz t^{-1} \to t\Nz''t^{-1} \to 1.
\end{equation*}
On veut exprimer l'action de Hecke de $T^+$ sur $\Hc(\Nz,V)$ à partir de celle (définie par \eqref{hecke} avec $\Nz'$ au lieu de $\Nz$) sur $\Hc(\Nz',V)$.

\begin{lemm}
Les $A$-modules $\Hc(\Nz',V)$ munis de l'action naturelle de $\Nz''$ et de l'action de Hecke de $T^+$ sont des représentations lisses de $\NT''$.
\end{lemm}

\begin{proof}
On montre d'abord le lemme en degré $0$. Il suffit de vérifier que l'action naturelle de $\Nz''$ et l'action de Hecke (notée $\h$) de $T^+$ sur $V^{\Nz'}$ sont compatibles : pour tous $t \in T^+, n'' \in \Nz'', v \in V^{\Nz'}$, on a
\begin{align*}
t \h (n'' \cdot v) &= \sum_{n' \in \Nz'/t\Nz't^{-1}} n'tn'' \cdot v \\
&= (tn''t^{-1}) \cdot \sum_{n' \in \Nz'/t\Nz't^{-1}} (tn''t^{-1})^{-1} n' (tn''t^{-1}) \cdot (t \cdot v) \\
&= (tn''t^{-1}) \cdot (t \h v),
\end{align*}
la dernière égalité résultant du fait que $t\Nz' t^{-1}$ est stable par conjugaison par $tn''t^{-1}$. En degré supérieur, on calcule les $A$-modules $\Hc(\Nz',V)$ à l'aide d'une résolution injective $V \hookrightarrow I^\bullet$ dans la catégorie des représentations lisses de $\NT$ sur $A$ (voir le lemme \ref{lemm:mnd} et \cite[Proposition 2.1.11]{Em2}) et on utilise le résultat en degré $0$ avec $(I^\bullet)^{\Nz'}$.
\end{proof}

Les foncteurs $\Hc(\Nz',-)$ forment donc un $\delta$-foncteur universel de la catégorie des représentations lisses de $\NT$ dans la catégorie des représentations lisses de $\NT''$.
On peut alors considérer l'action de Hecke de $t \in T^+$ (définie par \eqref{hecke} avec $\Nz''$ et $\Hc(\Nz',V)$ au lieu de $\Nz$ et $V$ respectivement) sur les $A$-modules $\Hc(\Nz'',\Hc(\Nz',V))$ : pour tous $i,j \in \Nbb$, c'est la composée
\begin{multline} \label{heckeHS}
\Hc[i](\Nz'',\Hc[j](\Nz',V)) \to \Hc[i](t\Nz''t^{-1},\Hc[j](t\Nz't^{-1},V)) \\
\to \Hc[i](t\Nz''t^{-1},\Hc[j](\Nz',V)) \to \Hc[i](\Nz'',\Hc[j](\Nz',V))
\end{multline}
où le premier morphisme est induit par l'action de $t$ sur $V$, le second par la corestriction de $t\Nz't^{-1}$ à $\Nz'$ et le dernier est la corestriction de $t\Nz''t^{-1}$ à $\Nz''$.
En degré $0$, le lemme suivant montre que l'on retrouve l'action de Hecke de $T^+$ définie par \eqref{hecke}.

\begin{lemm} \label{lemm:compcores}
L'action de Hecke de $t \in T^+$ sur $V^{\Nz}$ définie par \eqref{hecke} coïncide avec celle sur $(V^{\Nz'})^{\Nz''}$ définie par \eqref{heckeHS} avec $i=j=0$.
\end{lemm}

\begin{proof}
Soit $t \in T^+$. Soient $(n''_j)_{j \in \llbrack 1,l \rrbrack}$ et $(n'_i)_{i \in \llbrack 1,k \rrbrack}$ des systèmes de représentants des classes à gauche $\Nz''/t\Nz''t^{-1}$ et $\Nz'/t\Nz't^{-1}$ respectivement. Le produit de ces représentants induit une bijection
\begin{equation*}
\left( \Nz''/t\Nz''t^{-1} \right) \times \left( \Nz'/t\Nz't^{-1} \right) \iso \Nz/t\Nz t^{-1}.
\end{equation*}

Montrons l'injectivité. Soient $j_1,j_2 \in \llbrack 1,l \rrbrack$ et $i_1,i_2 \in \llbrack 1,k \rrbrack$. Si $n''_{j_1} n'_{i_1} \in n''_{j_2} n'_{i_2} t\Nz t^{-1}$, alors $\exists n'' \in t\Nz''t^{-1}, \exists n' \in t\Nz't^{-1}$ tels que $n''_{j_1} n'_{i_1} = n''_{j_2} n'_{i_2} n'' n'$, puis $n''_{j_1} n'_{i_1} = (n''_{j_2} n'') (n''^{-1} n'_{i_2} n'' n')$ d'où $n''_{j_1} = n''_{j_2} n''$ d'une part, donc $j_1=j_2$ et $n''=1$ ; et $n'_{i_1} = n''^{-1} n'_{i_2} n'' n' = n'_{i_2} n'$ d'autre part, donc $i_1=i_2$.

Montrons la surjectivité. Soit $n \in \Nz$. Si $n=n''n'$ avec $n'' \in \Nz''$ et $n' \in \Nz'$, alors $\exists j \in \llbrack 1,l \rrbrack, \exists m'' \in t\Nz''t^{-1}$ tels que $n''=n''_j m''$, puis $\exists i \in \llbrack 1,k \rrbrack, \exists m' \in t\Nz't^{-1}$ tels que $m''n'm''^{-1}=n'_i m'$ ; ainsi $n''n' = n''_j n'_i m'm''$ donc $n \in n''_jn'_it\Nz t^{-1}$.

En utilisant ces systèmes de représentants, on voit que la corestriction de $t\Nz t^{-1}$ à $\Nz$ est la composée de celle de $t\Nz't^{-1}$ à $\Nz'$ et de celle de $t\Nz''t^{-1}$ à $\Nz''$ : pour tout $v \in V^{\Nz}$ on a
\begin{equation*}
\sum_{n \in \Nz/t\Nz t^{-1}} n \cdot (t \cdot v) = \sum_{n'' \in \Nz''/t\Nz''t^{-1}} n'' \cdot \sum_{n' \in \Nz'/t\Nz't^{-1}} n' \cdot (t \cdot v).
\end{equation*}
L'action de Hecke de $t$ sur $V^{\Nz}$ définie par \eqref{hecke} coïncide donc avec la composée
\begin{equation*}
\left( V^{\Nz'} \right)^{\Nz''} \to \left( V^{t\Nz't^{-1}} \right)^{t\Nz''t^{-1}} \to \left( V^{\Nz'} \right)^{t\Nz''t^{-1}} \to \left( V^{\Nz'} \right)^{\Nz''}
\end{equation*}
où le premier morphisme est induit par l'action de $t$ sur $V$, le second par la corestriction de $t\Nz't^{-1}$ à $\Nz'$ et le dernier est la corestriction de $t\Nz''t^{-1}$ à $\Nz''$ ; c'est-à-dire avec l'action de $t$ sur $(V^{\Nz'})^{\Nz''}$ définie par \eqref{heckeHS} avec $i=j=0$.
\end{proof}

En degré supérieur, la proposition suivante montre que la suite spectrale \eqref{HS} est compatible avec les actions de Hecke de $T^+$ définies sur $\Hc[i+j](\Nz,V)$ par \eqref{hecke} et sur $\Hc[i](\Nz'',\Hc[j](\Nz',V))$ par \eqref{heckeHS}.

\begin{prop} \label{prop:comphecke}
La suite spectrale \eqref{HS} est définie dans la catégorie des représentations lisses de $T^+$ sur $A$.
\end{prop}

\begin{proof}
On note $\Mod(A)$ la catégorie des modules sur $A$ et $\Rep(-)$ les catégories de représentations lisses sur $A$. On note $\mathcal{F}$ et $\mathcal{\mathcal{G}}$ (resp. $\mathcal{F}'$ et $\mathcal{G}'$, $\mathcal{F}''$ et $\mathcal{G}''$) les foncteurs des $\Nz$-invariants (resp. $\Nz'$-invariants, $\Nz''$-invariants) et $\For$ les différents foncteurs d'oubli. On précise cela dans le diagramme de catégories et de foncteurs suivant.
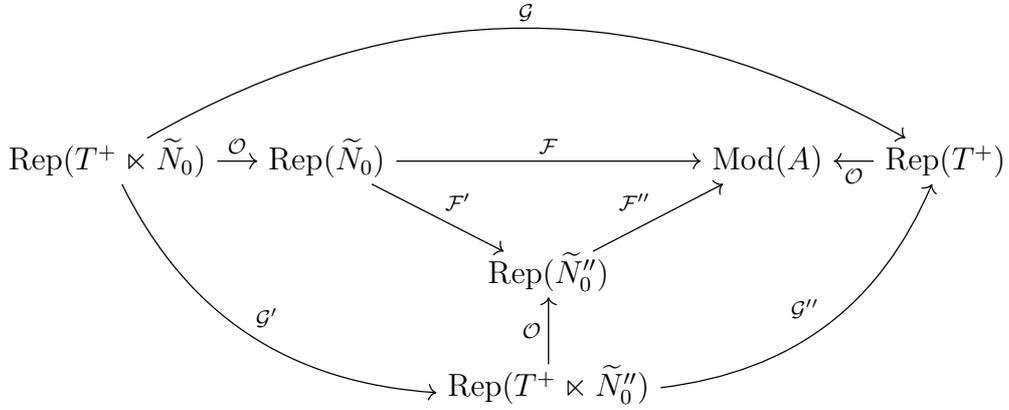
\begin{equation*} \begin{tikzcd}[column sep=small]
\Rep(\NT) \rar{\For} \ar[bend left]{rrrr}{\mathcal{G}} \ar[bend right]{rrdd}{\mathcal{G}'} & \Rep(\Nz) \ar{rd}{\mathcal{F}'} \ar{rr}{\mathcal{F}} &&  \Mod(A) & \Rep(T^+) \lar{\For} \\
&& \Rep(\Nz'') \ar{ur}{\mathcal{F}''} && \\
&& \Rep(\NT'') \uar{\For} \ar[bend right]{rruu}{\mathcal{G}''} &&
\end{tikzcd} \end{equation*}
Ce diagramme est commutatif : le lemme \ref{lemm:compcores} montre que $\mathcal{G} = \mathcal{G}'' \circ \mathcal{G}'$ et les autres relations sont immédiates. Les foncteurs d'oubli sont exacts et préservent les injectifs (voir le lemme \ref{lemm:mnd} et \cite[Proposition 2.1.11]{Em2}). Ainsi pour tout $n \in \Nbb$, on a des isomorphismes de foncteurs
\begin{equation*}
\Rc[n]\mathcal{F} \circ \For \cong \Rc[n](\mathcal{F} \circ \For) \cong \Rc[n](\For \circ \mathcal{G}) \cong \For \circ \Rc[n]\mathcal{G}
\end{equation*}
avec $\Rc$ les foncteurs dérivés à droite. De même pour $\mathcal{F}',\mathcal{G}'$ et $\mathcal{F}'',\mathcal{G}''$. On en déduit pour tous $i,j \in \Nbb$, des isomorphismes de foncteurs
\begin{equation*}
\Rc[i]\mathcal{F}'' \circ \Rc[j]\mathcal{F}' \circ \For \cong \For \circ \Rc[i]\mathcal{G}'' \circ \Rc[j]\mathcal{G}'.
\end{equation*}
Par ailleurs si $I$ est un objet injectif de $\Rep(\NT)$, alors pour tout entier $n>0$ on a
\begin{equation*}
\For(\Rc[n]\mathcal{G}'' (\mathcal{G}'(I)))=\Rc[n]\mathcal{F}''(\mathcal{F}'(\For(I)))=0
\end{equation*}
car $\mathcal{F}'$ préserve les injectifs (c'est l'adjoint à droite du foncteur exact d'inflation). Ainsi $\mathcal{G}'$ envoie les injectifs sur des $\mathcal{G}''$-acycliques. On en déduit l'existence d'une suite spectrale de Grothendieck
\begin{equation*}
\Rc[i]\mathcal{G}''(\Rc[j]\mathcal{G}'(V)) \Rightarrow \Rc[i+j]\mathcal{G}(V)
\end{equation*}
et les isomorphismes de foncteurs précédents montrent que l'image de cette suite spectrale par $\For$ est précisément la suite spectrale de Grothendieck associée aux foncteurs $\mathcal{F}$, $\mathcal{F}'$ et $\mathcal{F}''$, c'est-à-dire la suite spectrale \eqref{HS}.
\end{proof}

\subsection{Calculs sur le gradué de la filtration} \label{sec:calc}

Soient $U$ une représentation lisse de $T(F)$ sur $A$ et $w \in W$. On suppose que $N_0$ est donné par le lemme \ref{lemm:compat} avec $c$ suffisamment grand pour que $p^c$ soit nul dans $A$. On note
\begin{equation*}
V_w \dfn \Clis_c(N_w(F),U)
\end{equation*}
la représentation lisse de $B(F)$ sur $A$ définie par l'isomorphisme \eqref{isoClis}.
Dans cette sous-section, on calcule la cohomologie de $N_0$ à valeurs dans $V_w$ et l'action de Hecke de $T^+$ sur les $A$-modules $\Hc(N_0,V_w)$.

\subsubsection*{Simplification préliminaire}

On fixe un élément $\varsigma \in S^+$ comme dans le lemme \ref{lemm:tau} (en procédant comme dans la preuve du lemme, un tel élément peut être choisi dans $S^+$ d'après la remarque \ref{rema:S}). Pour tout $k \in \Nbb$, on pose
\begin{equation*}
N_{w,k} \dfn N_w(F) \cap (\varsigma^{-k} N_0 \varsigma^k) \quad \text{et} \quad V_{w,k} \dfn \{f \in V_w ~|~ \supp(f) \subset N_{w,k}\}.
\end{equation*}
Par définition de $\varsigma$ on a $N_w(F) = \bigcup_{k \geq 0} \varsigma^{-k} \Nwz \varsigma^k$, d'où
\begin{equation} \label{union}
V_w = \bigcup_{k \geq 0} V_{w,k}.
\end{equation}
On déduit de l'isomorphisme \eqref{prodNa} que $N_{w_0w}(F) N_w(F) = N_0$, donc d'après la proposition \ref{prop:rad} on a $N_0 = (N_{w_0w}(F) \cap N_0) \Nwz$, d'où
\begin{equation} \label{BwN}
B^-(F)\w N_0 = B^-(F)\w\Nwz.
\end{equation}
On remarque enfin que l'on a un isomorphisme $\Nwz$-équivariant
\begin{equation} \label{isoVwz}
\Vwz \cong \Clis(\Nwz,U)
\end{equation}
où $\Nwz$ agit sur $\Clis(\Nwz,U)$ par translation à droite. Le sous-$A$-module $\Vwz$ de $V_w$ est stable par $N_0$ et $T^+$ : l'égalité \eqref{BwN} montre que $\Vwz$ est stable par $N_0$, tandis que la description de l'action de $T(F)$ sur $V_w$ donnée après l'isomorphisme \eqref{isoClis} montre que $\Vwz$ est stable par $T^+$. C'est donc une représentation lisse de $T^+ \ltimes N_0$ sur $A$ et de même pour le $A$-module quotient $V_w/\Vwz$.

\begin{lemm} \label{lemm:V0}
On a des suites exactes courtes de représentations lisses de $T^+$ sur $A$
\begin{equation*}
0 \to \Hc(N_0,\Vwz) \to \Hc(N_0,V_w) \to \Hc(N_0,V_w/\Vwz) \to 0
\end{equation*}
et l'action de Hecke de $\varsigma$ sur $\Hc(N_0,V_w/\Vwz)$ est localement nilpotente.
\end{lemm}

\begin{proof}
On a une suite exacte courte de représentations lisses de $N_0$ sur $A$
\begin{equation*}
0 \to \Vwz \to V_w \to V_w/\Vwz \to 0
\end{equation*}
qui est scindée : la multiplication par la fonction caractéristique de $\Nwz$ sur $N_w(F)$ induit une rétraction $A$-linéaire $V_w \to \Vwz$ et l'égalité \eqref{BwN} montre qu'elle est $N_0$-équivariante.
En prenant la cohomologie de $N_0$ à valeurs dans cette suite exacte, on obtient donc des suites exactes courtes de représentations lisses de $T^+$ sur $A$
\begin{equation*}
0 \to \Hc(N_0,\Vwz) \to \Hc(N_0,V_w) \to \Hc(N_0,V_w/\Vwz) \to 0.
\end{equation*}
Pour tout entier $k>0$, l'image de $V_{w,k}$ sous l'action de $\varsigma$ est incluse dans $V_{w,k-1}$. En utilisant l'égalité \eqref{union}, on en déduit que l'action de $\varsigma$ sur $V_w/\Vwz$ est localement nilpotente. Soient $n \in \Nbb$ et $\phi : N_0^n \to V_w/\Vwz$ une cochaîne. Comme $\phi$ est localement constante et $N_0^n$ est compact, $\im(\phi) \subset V_w/\Vwz$ est fini donc annulé par une puissance suffisamment grande de $\varsigma$. On en déduit que la classe de $\phi$ dans $\Hc[n](N_0,V_w/\Vwz)$ est annulée par l'action de Hecke de cette même puissance de $\varsigma$. Ainsi l'action de Hecke de $\varsigma$ sur $\Hc[n](N_0,V_w/\Vwz)$ est localement nilpotente.
\end{proof}

\subsubsection*{Calculs par récurrence}

On fixe une décomposition réduite $s_{\ell(w)} \dots s_1$ de $w$. Soit $k \in \llbrack 0,\ell(w) \rrbrack$. À travers la décomposition radicielle \eqref{rad} de $\Lie(N)$, le sous-groupe $\Ns{k}$ (défini par \eqref{Nw} avec $s_k \dots s_1$ au lieu de $w$) de $N$ correspond au sous-ensemble de $\Phi^+$ suivant
\begin{equation*}
\Phi_{s_k \dots s_1}^+ \dfn \{\alpha \in \Phi^+ ~|~ s_k \dots s_1 (\alpha) \in \Phi^+\}
\end{equation*}
au sens où son algèbre de Lie est la somme directe des sous-espaces propres des poids appartenant à $\Phi_{s_k \dots s_1}^+$.
On suppose $k<\ell(w)$ et soit $\alpha \in \Delta$ tel que $s_{k+1}=s_\alpha$. On déduit de \cite[Partie II, § 1.5]{Jan} que
\begin{equation*}
\Phi_{s_k \dots s_1}^+ = \Phi_{s_{k+1} \dots s_1}^+ \amalg \{s_1 \dots s_k(\alpha)\}.
\end{equation*}
Ainsi $\Ns{k+1}$ est un sous-groupe fermé de $\Ns{k}$ de codimension $1$, donc il est distingué dans $\Ns{k}$ d'après \cite[Chapitre IV, § 4, Corollaire 1.9]{DG}. En utilisant l'isomorphisme \eqref{prodNa}, on voit que le quotient
\begin{equation*}
\Ns{k}'' \dfn \Ns{k} / \Ns{k+1}
\end{equation*}
s'identifie au sous-groupe radiciel de $N$ correspondant à la racine $s_1 \dots s_k(\alpha)$. On a donc un produit semi-direct $\Ns{k} = \Ns{k}'' \ltimes \Ns{k+1}$. On note $\Nsz{k}$ et $\Nsz{k}''$ les intersections respectives de $\Ns{k}(F)$ et $\Ns{k}''(F)$ avec $N_0$. Ce sont des sous-groupes ouverts compacts standards de $\Ns{k}(F)$ et $\Ns{k}''(F)$ respectivement stables sous l'action par conjugaison de $T^+$. En utilisant la proposition \ref{prop:rad}, on voit que l'on a une suite exacte courte scindée de groupes topologiques
\begin{equation} \label{devk}
1 \to \Nsz{k+1} \to \Nsz{k} \to \Nsz{k}'' \to 1.
\end{equation}
En utilisant la proposition \ref{prop:comphecke} avec ce dévissage, on voit que l'on a une suite spectrale de représentations lisses de $T^+$ sur $A$
\begin{equation} \label{HSk}
\Hc[i](\Nsz{k}'',\Hc[j](\Nsz{k+1},\Vwz)) \Rightarrow \Hc[i+j](\Nsz{k},\Vwz).
\end{equation}

\begin{lemm} \label{lemm:H1}
Soient $k \in \llbrack 0,\ell(w) \rrbrack$ et $n \in \Nbb$. Si $n > [F:\Qp] \cdot (\ell(w)-k)$, alors $\Hc[n](\Nsz{k},\Vwz)=0$.
\end{lemm}

\begin{proof}
On suppose $n > [F:\Qp] \cdot (\ell(w)-k)$ et on procède par récurrence décroissante sur $k$. On suppose $k=\ell(w)$, donc $n>0$. En utilisant le lemme \ref{lemm:inj}, on déduit de l'isomorphisme \eqref{isoVwz} que $\Vwz$ est $\Nwz$-acyclique, donc $\Hc[n](\Nwz,\Vwz)=0$.

On suppose $k<\ell(w)$ et le lemme vrai pour $k+1$. Soient $i,j \in \Nbb$ tels que $i+j=n$. Si $i > [F:\Qp]$, alors $\Hc[i](\Nsz{k}'',\Hc[j](\Nsz{k+1},\Vwz))=0$ d'après le lemme \ref{lemm:Hnul}. Sinon $j > [F:\Qp] \cdot (\ell(w)-(k+1))$, donc $\Hc[j](\Nsz{k+1},\Vwz)=0$ par hypothèse de récurrence. Dans les deux cas, on en conclut que
\begin{equation*}
\Hc[i](\Nsz{k}'',\Hc[j](\Nsz{k+1},\Vwz))=0.
\end{equation*}
En utilisant la suite spectrale \eqref{HSk}, on en déduit que $\Hc[n](\Nsz{k},\Vwz)=0$. Donc le lemme est vrai pour $k$.
\end{proof}

On note $U^w$ la représentation lisse de $T(F)$ sur $A$ dont le $A$-module sous-jacent est $U$ et sur lequel $t \in T(F)$ agit à travers $\w t \w^{-1}$. Pour tout $k \in \llbrack 0,\ell(w) \rrbrack$, on définit un caractère algébrique de $T$ en posant
\begin{equation*}
\alpha_k \dfn \sum_{\alpha \in \Phi^+_{s_k\dots s_1} - \Phi^+_w} \alpha.
\end{equation*}

\begin{lemm} \label{lemm:H2}
Soient $k \in \llbrack 0,\ell(w) \rrbrack$ et $n \in \Nbb$. On suppose ou bien $A=\ke$ et $U$ de dimension $1$, ou bien $n \leq 1$. Si $n = [F:\Qp] \cdot (\ell(w)-k)$, alors on a un isomorphisme $T^+$-équivariant
\begin{equation*}
\Hc[n](\Nsz{k},\Vwz) \cong U^w \otimes (\oma_k)
\end{equation*}
Si de plus $k>0$, alors l'action naturelle de $\Nsz{k-1}''$ sur $\Hc[n](\Nsz{k},\Vwz)$ est triviale.
\end{lemm}

\begin{proof}
On suppose $n=[F:\Qp] \cdot (\ell(w)-k)$ et on procède par récurrence décroissante sur $k$.
On suppose $k=\ell(w)$, donc $n=0$. L'évaluation en $1 \in N_w(F)$ induit un isomorphisme $A$-linéaire $\Vwz^{\Nwz} \iso U$. On calcule l'action de Hecke (notée $\h$) de $T^+$ sur $\Vwz^{\Nwz}$ à travers cet isomorphisme : pour tous $t \in T^+, f \in \Vwz^{\Nwz}$, on a
\begin{align*}
(t \h f)(1) &= \sum_{n \in \Nwz/t\Nwz t^{-1}} (nt f) (1) \\
&= \sum_{n \in \Nwz/t\Nwz t^{-1}} (\w t\w^{-1}) \cdot f(t^{-1}nt) \\
&= (\w t\w^{-1}) \cdot f(1),
\end{align*}
la dernière égalité résultant du fait que $t^{-1}nt \in \Nwz$ si et seulement si $n \in t \Nwz t^{-1}$. On a donc un isomorphisme $T^+$-équivariant $\Vwz^{\Nwz} \cong U^w$ qui est bien celui de l'énoncé car $\alpha_{\ell(w)}=0$. De plus on voit grâce à l'isomorphisme \eqref{isoClis} et l'égalité \eqref{BwN} que $\Vwz^{\Nwz}=\Vwz^{N_0}$ et on en déduit que si $\ell(w)>0$, alors l'action de $\Nsz{\ell(w)-1}''$ sur $\Vwz^{\Nwz}$ est triviale.

On suppose $k<\ell(w)$ et le lemme vrai pour $k+1$. Soient $i,j \in \Nbb$ tels que $i+j=n$. Si $i > [F:\Qp]$, alors $\Hc[i](\Nsz{k}'',\Hc[j](\Nsz{k+1},\Vwz)) = 0$ d'après le lemme \ref{lemm:Hnul}. Si $j > [F:\Qp] \cdot (\ell(w)-(k+1))$, alors $\Hc[j](\Nsz{k+1},\Vwz)=0$ d'après le lemme \ref{lemm:H1}. En utilisant la suite spectrale \eqref{HSk}, on en déduit un isomorphisme $T^+$-équivariant
\begin{equation*}
\Hc[n](\Nsz{k},\Vwz) \cong \Hc[i](\Nsz{k}'',\Hc[j](\Nsz{k+1},\Vwz))
\end{equation*}
avec $i = [F:\Qp]$ et $j = [F:\Qp] \cdot (\ell(w)-(k+1))$. Dans ce cas par hypothèse de récurrence (si $n \leq 1$, alors $j \leq 1$), on a un isomorphisme $T^+$-équivariant
\begin{equation*}
\Hc[j](\Nsz{k+1},\Vwz) \cong U^w \otimes (\oma_{k+1})
\end{equation*}
et $\Nsz{k}''$ agit trivialement sur $\Hc[j](\Nsz{k+1},\Vwz)$. En utilisant la proposition \ref{prop:dimN}, on en déduit un isomorphisme $T^+$-équivariant
\begin{equation*}
\Hc[n](\Nsz{k},\Vwz) \cong U^w \otimes (\omega^{-1} \circ (\alpha_{k+1}+\alpha))
\end{equation*}
avec $\alpha$ le caractère de la représentation adjointe de $T$ sur $\Lie(\Ns{k}'')$. On conclut en utilisant le fait que $\alpha_{k+1} + \alpha = \alpha_k$.

Si $k>0$, alors il reste à montrer que $\Nsz{k-1}''$ agit trivialement sur $\Hc[n](\Nsz{k},\Vwz)$. Dans ce cas on utilise l'hypothèse de l'énoncé : ou bien $A=\ke$ et $U$ est de dimension $1$, auquel cas $\Hc[n](\Nsz{k},\Vwz)$ est un $\ke$-espace vectoriel de dimension $1$, donc l'action lisse du pro-$p$-groupe $\Nsz{k-1}''$ est triviale, ou bien $n \leq 1$. On suppose $n=1$ (le cas $n=0$ ayant déjà été traité), donc $F=\Qp$ et $k=\ell(w)-1$. En utilisant le lemme \ref{lemm:H1} et la suite exacte \eqref{HSSE} avec $n=1$ et le dévissage \eqref{devk}, on voit que l'inflation induit un isomorphisme $A$-linéaire
\begin{equation*}
\Hc[1](\Nsz{k}'',\Vwz^{\Nwz}) \iso \Hc[1](\Nsz{k},\Vwz).
\end{equation*}
On note $[-] : \Nsz{k} \twoheadrightarrow \Nsz{k}''$ et $[-] : \Lie(\Nsz{k}) \twoheadrightarrow \Lie(\Nsz{k}'')$ les applications quotient, de sorte que $\log([n])=[\log(n)]$ pour tout $n \in \Nsz{k}$. Soient $\phi : \Nsz{k}'' \to \Vwz^{\Nwz}$ un cocycle, donc un morphisme de groupes continu (car $\Nsz{k}''$ agit trivialement sur $\Vwz^{\Nwz}$) et $\widetilde{\phi} : \Nsz{k} \to \Vwz$ le cocycle obtenu par inflation. Soient $n'' \in \Nsz{k-1}''$ et $n \in \Nsz{k}$. Alors
\begin{equation*}
(n'' \tilde{\phi})(n) = n'' \cdot \tilde{\phi}(n''^{-1}nn'') = n'' \cdot \phi([n''^{-1}nn'']) = \phi([n''^{-1}nn'']),
\end{equation*}
la dernière égalité résultant du fait que $\im \phi \subset \Vwz^{\Nwz}=\Vwz^{N_0}$, puis
\begin{equation*}
\phi([n''^{-1}nn'']) = (\phi \circ \exp)(\log [n''^{-1}nn'']) = (\phi \circ \exp)([\log (n''^{-1}nn'')]).
\end{equation*}
Comme $\Nsz{k}''$ est abélien, $\exp$ est aussi un morphisme de groupes continu, donc $\phi \circ \exp$ est $\Zp$-linéaire. Or $N_0$ étant donné par le lemme \ref{lemm:compat}, on a
\begin{equation*}
\log (n''^{-1}nn'') \equiv \log n \pmod{p^c}
\end{equation*}
d'où, par $\Zp$-linéarité et puisque $p^c=0$ dans $A$ par hypothèse, l'égalité
\begin{equation*}
\phi([n''^{-1}nn'']) = (\phi \circ \exp)([\log n]) = \phi([n]).
\end{equation*}
Ainsi $n''\tilde{\phi}=\tilde{\phi}$ et l'action de $\Nsz{k-1}''$ sur $H^1(\Nsz{k},\Vwz)$ est triviale. Donc le lemme est vrai pour $k$.
\end{proof}

\begin{lemm} \label{lemm:H3}
Soient $k \in \llbrack 0,\ell(w) \rrbrack$ et $n \in \Nbb$. On suppose ou bien $A=\ke$ et $U$ de dimension $1$, ou bien $n \leq 1$. Si $n < [F:\Qp] \cdot (\ell(w)-k)$, alors l'action de Hecke de $\varsigma$ sur $\Hc[n](\Nsz{k},\Vwz)$ est nilpotente.
\end{lemm}

\begin{proof}
On suppose $n < [F:\Qp] \cdot (\ell(w)-k)$ et on procède par récurrence décroissante sur $k$. Pour $k=\ell(w)$ il n'y a rien à vérifier. On suppose $k<\ell(w)$ et le lemme vrai pour $k+1$.

Soient $i,j \in \Nbb$ tels que $i+j=n$. Si $j < [F:\Qp] \cdot (\ell(w)-(k+1))$, alors l'action Hecke de $\varsigma$ sur $\Hc[j](\Nsz{k+1},\Vwz)$ est nilpotente par hypothèse de récurrence (si $n \leq 1$, alors $j \leq 1$). Si $j > [F:\Qp] \cdot (\ell(w)-(k+1))$, alors $\Hc[j](\Nsz{k+1},\Vwz)=0$ d'après le lemme \ref{lemm:H1}. Dans les deux cas, on en conclut que l'action de Hecke de $\varsigma$ sur $\Hc[i](\Nsz{k}'',\Hc[j](\Nsz{k+1},\Vwz))$ est nilpotente. On suppose $j = [F:\Qp] \cdot (\ell(w)-(k+1))$, donc $i <[F:\Qp]$.

Comme $\Res \Ns{k}''$ est unipotent et commutatif, il est isomorphe au produit direct de $[F:\Qp]$ copies du groupe additif sur $\Qp$ :
\begin{equation} \label{devl}
\Res \Ns{k}'' \cong \N_1 \times \dots \times \N_{[F:\Qp]}.
\end{equation}
En utilisant la remarque \ref{rema:S}, on voit que cette décomposition est stable sous l'action par conjugaison de $S$.
Pour tout $l \in \llbrack 1,[F:\Qp] \rrbrack$, on note $\N_{l,0}$ l'intersection de $\N_l(\Qp)$ avec $N_0$ et on pose
\begin{equation*}
\Plz{l} \dfn \N_{1,0} \times \dots \times \N_{l,0} \quad \text{et} \quad \V \dfn \Hc[j](\Nsz{k+1},\Vwz).
\end{equation*}

Soit $l \in \llbrack 1,[F:\Qp] \rrbrack$. On montre que si $i<l$, alors l'action de Hecke de $\varsigma$ sur $\Hc[i](\Plz{l},\V)$ définie dans la remarque \ref{rema:heckeS} est nilpotente.
On suppose $i=0$ et $l>0$. L'action de $\Plz{l}$ sur $\V$ étant triviale d'après le lemme \ref{lemm:H2}, l'action de Hecke (notée $\h$) de $\varsigma$ sur $v \in \V^{\Plz{l}}$ est donnée par
\begin{equation*}
\varsigma \h v = \left( \Plz{l} : \varsigma \Plz{l} \varsigma^{-1} \right) \left( \varsigma \cdot  v \right) = |\alpha(\varsigma)|_p^{-l} \left( \varsigma \cdot  v \right)
\end{equation*}
avec $\alpha$ le caractère de la représentation adjointe de $T$ sur $\Lie(\Ns{k}'')$. Comme $A$ est artinien et $\val(\alpha(\varsigma))>0$ par définition, cette action est nilpotente.

On procède par récurrence sur $l$. Le cas $i=0$ (en particulier lorsque $l=1$) a déjà été traité. On suppose donc $0<i<l$ et le résultat vrai pour $l-1$. En utilisant la suite exacte \eqref{HSSE} et la remarque \ref{rema:heckeS}, on voit que la décomposition \eqref{devl} induit une suite exacte courte de représentations lisses de $S^+$ sur $A$
\begin{equation*}
 0 \to \Hc[1](\N_{l,0},\Hc[i-1](\Plz{l-1},\V)) \to \Hc[i](\Plz{l},\V) \to \Hc[i](\Plz{l-1},\V)^{\N_{l,0}} \to 0.
\end{equation*}
D'un côté $i-1<l-1$, donc par hypothèse de récurrence l'action de Hecke de $\varsigma$ sur $\Hc[i-1](\Plz{l-1},\V)$ est nilpotente, donc sur $\Hc[1](\N_{l,0},\Hc[i-1](\Plz{l-1},\V))$ aussi. De l'autre l'action de $\N_{l,0}$ sur $\Hc[i](\Plz{l-1},\V)$ est triviale (car $\N_{l,0}$ centralise $\Plz{l-1}$ et agit trivialement sur $\V$ d'après le lemme \ref{lemm:H2}), donc (comme dans le cas $i=0$) l'action de Hecke de $\varsigma$ sur $\Hc[i](\Plz{l-1},\V)^{\N_{l,0}}$ est nilpotente. On en conclut que l'action de Hecke de $\varsigma$ sur $\Hc[i](\Plz{l},\V)$ est nilpotente. Donc le résultat est vrai pour $l$.

En particulier avec $l=[F:\Qp]$, on a montré que l'action de Hecke de $\varsigma$ sur $\Hc[i](\Nsz{k}'',\Hc[j](\Nsz{k+1},\Vwz))$ est nilpotente. En utilisant la suite spectrale \eqref{HSk}, on en déduit que l'action de Hecke de $\varsigma$ sur $\Hc[n](\Nsz{k},\Vwz)$ est nilpotente.
Donc le lemme est vrai pour $k$.
\end{proof}

\section{Parties ordinaires dérivées d'une induite} \label{sec:4}

Nous rappelons la construction et les propriétés du $\delta$-foncteur $\HOrd$. Puis, nous calculons les foncteurs $\HOrd$ sur certaines représentations induites de $G(F)$ sur $A$. Enfin, nous menons les calculs analogues pour les induites paraboliques d'un caractère.

\subsection{\texorpdfstring{Le $\delta$-foncteur $\HOrd$}{Le delta-foncteur HOrd}}

On rappelle la construction du $\delta$-foncteur $\HOrd$ ainsi que ses propriétés en suivant \cite[§ 3]{Em2}. En particulier, on rappelle les résultats dérivés de la relation d'adjonction entre les foncteurs $\Ind$ et $\Ord$.

\subsubsection*{Construction}

Soit $N_0$ un sous-groupe ouvert compact standard de $N(F)$ compatible avec la décomposition radicielle. Si $V$ est une représentation lisse de $B(F)$ sur $A$, alors on définit des représentations lisses de $T(F)$ sur $A$, que l'on appellera les \emph{parties ordinaires dérivées} de $V$, en faisant agir $T(F)$ par translation à droite sur le $A$-module
\begin{equation*}
\HOrd(V) \dfn \HomTp{\Hc(N_0,V)}
\end{equation*}
où l'indice $_{\fin{T(F)}}$ désigne les éléments qui engendrent un sous-$A$-module de type fini sous l'action de $T(F)$. Les foncteurs $\HOrd$ ainsi définis ne dépendent ni du choix de $T$ ni du choix de $N_0$. Ils sont nuls en degré strictement plus grand que $d$ et forment un $\delta$-foncteur de la catégorie des représentations lisses localement admissibles de $G(F)$ sur $A$ dans la catégorie des représentations lisses localement admissibles de $T(F)$ sur $A$. 

On rappelle le résultat suivant (voir \cite[Lemme 3.2.1]{Em2}) qui montre que le foncteur $\HomTp{-}$ est une localisation lorsqu'on le restreint à une sous-catégorie particulière.

\begin{lemm} \label{lemm:x}
Soit $\mathcal{X}$ la sous-catégorie pleine des représentations lisses de $T^+$ sur $A$ qui sont réunion de sous-$A$-modules de type fini $T^+$-invariants. Alors pour tout $X\in\mathcal{X}$, on a un isomorphisme naturel $T(F)$-équivariant
\begin{equation*}
\HomTp{X} \cong A[T(F)] \otimes_{A[T^+]} X.
\end{equation*}
En particulier, ce foncteur est exact sur $\mathcal{X}$ et si l'action de $t \in T^+$ sur $X \in \mathcal{X}$ est localement nilpotente, alors $\HomTp{X}=0$.
\end{lemm}

\begin{rema} \label{rema:xab}
La sous-catégorie $\mathcal{X}$ est stable par passage aux sous-objets (car $A$ est noethérien) et aux objets quotients.
\end{rema}

\subsubsection*{Adjonction dérivée}

Soient $U$ et $V$ des représentations lisses localement admissibles de $T(F)$ et $G(F)$ respectivement sur $A$. On déduit de \cite[Théorème 4.4.6]{Em1} que le foncteur $\Ord$ induit un isomorphisme
\begin{equation*}
\Hom_{G(F)} \left( \Ind U,V \right) \cong \Hom_{T(F)} \left( U,\Ord V \right).
\end{equation*}
Comme $\Ord$ envoie les injectifs de la catégorie des représentations lisses localement admissibles de $G(F)$ sur $A$ sur des injectifs de la catégorie des représentations lisses localement admissibles de $T(F)$ sur $A$ (puisque c'est l'adjoint à droite du foncteur exact $\Ind$), on a une suite spectrale de Grothendieck de $A$-modules
\begin{equation} \label{Gro}
\Ext_{T(F)}^i \left( U,\ROrd[j] V \right) \Rightarrow \Ext_{G(F)}^{i+j} \left( \Ind U,V \right)
\end{equation}
où $\ROrd$ désigne les foncteurs dérivés à droite du foncteur $\Ord$ dans la catégorie des représentations lisses localement admissibles de $G(F)$ sur $A$. Comme ces derniers forment un $\delta$-foncteur universel qui coïncide avec $\HOrd$ en degré $0$, on a un morphisme de $\delta$-foncteurs
\begin{equation} \label{ROrd}
\ROrd \to \HOrd
\end{equation}
qui est un isomorphisme en degré $0$, donc une injection en degré $1$
\begin{equation} \label{R1Ord}
\ROrd[1] \hookrightarrow \HOrd[1].
\end{equation}
De plus, on a la conjecture suivante (voir \cite[Conjecture 3.7.2]{Em2}) qui est démontrée pour $\GL_2$ dans \cite{EmP}.

\begin{conj}[Emerton] \label{conj:ROrd}
Le morphisme \eqref{ROrd} est un isomorphisme.
\end{conj}

Les termes de bas degré de la suite spectrale \eqref{Gro} donnent une suite exacte de $A$-modules
\begin{multline*}
0 \to \Ext_{T(F)}^1 \left( U,\Ord V \right) \to \Ext_{G(F)}^1 \left( \Ind U,V \right) \\
\to \Hom_{T(F)} \left( U,\ROrd[1] V \right) \to \Ext_{T(F)}^2 \left( U,\Ord V \right) \\
\to \Ext_{G(F)}^2 \left( \Ind U,V \right).
\end{multline*}
En utilisant l'injection \eqref{R1Ord}, on obtient donc une suite exacte de $A$-modules
\begin{multline} \label{GroSE}
0 \to \Ext_{T(F)}^1 \left( U,\Ord V \right) \to \Ext_{G(F)}^1 \left( \Ind U,V \right) \\
\to \Hom_{T(F)} \left( U,\HOrd[1] V \right).
\end{multline}
Le premier morphisme non trivial de cette suite exacte est induit par le foncteur $\Ind$ et le second est induit par le morphisme connectant en degré $0$ du $\delta$-foncteur $\HOrd$.

\subsection{Calculs de parties ordinaires dérivées}

On montre que la filtration de Bruhat d'une représentation induite de $G(F)$ sur $A$ induit une filtration de ses parties ordinaires dérivées. On en déduit l'expression du $\delta$-foncteur $\HOrd$ sur $\Ind U$ en degré $1$ pour toute représentation lisse localement admissible $U$ de $T(F)$ sur $A$ et en degré quelconque lorsque $U$ est un caractère et $A=\ke$.

\subsubsection*{Filtration et calcul sur le gradué}

\begin{lemm} \label{lemm:filx}
Soit $U$ une représentation lisse localement admissible de $T(F)$ sur $A$. Pour tout $r \in \llbrack 0,d \rrbrack$, les suites exactes \eqref{filcohom} sont dans la catégorie $\mathcal{X}$ du lemme \ref{lemm:x}.
\end{lemm}

\begin{proof}
Soit $r \in \llbrack 0,d \rrbrack$. Les morphismes des suites exactes \eqref{filind} sont $B(F)$-équivariants, donc ceux des suites exactes \eqref{filcohom} sont $T^+$-équivariants pour l'action de Hecke. Comme $U$ est localement admissible il en est de même pour $\Ind U$, donc les $A$-modules $\Hc(N_0,\Ind U)$ sont dans $\mathcal{X}$ d'après \cite[Théorème 3.4.7]{Em2}. En passant aux sous-objets puis aux objets quotients (voir la remarque \ref{rema:xab}), on en déduit que les suites exactes \eqref{filcohom} sont également dans $\mathcal{X}$.
\end{proof}

En utilisant les lemmes \ref{lemm:x} et \ref{lemm:filx}, on voit que si $U$ est une représentation lisse localement admissible de $T(F)$ sur $A$, alors pour tout $r \in \llbrack 0,d \rrbrack$ on a des suites exactes courtes de représentations lisses localement admissibles de $T(F)$ sur $A$
\begin{multline} \label{filhord}
0 \to \HOrd (I_{r-1}) \to \HOrd (I_r) \\
\to \bigoplus_{\ell(w)=r} \HOrd(\Clis_c(N_w(F),U)) \to 0.
\end{multline}

On rappelle que si $U$ est une représentation lisse de $T(F)$ sur $A$, alors pour tout $w \in W$ on note $U^w$ (ou simplement $U^\alpha$ lorsque $w=s_\alpha$ avec $\alpha \in \Delta$) la représentation de $T(F)$ dont le $A$-module sous-jacent est $U$ et sur lequel $t \in T(F)$ agit à travers $\w t \w^{-1}$. Lorsque $U$ est un caractère $\chi$ on note $w(\chi)=\chi^{w^{-1}}$. On note enfin $\alpha_w$ le caractère algébrique de la représentation adjointe de $T$ sur $\det_F \Lie(N_{w_0w})$.

\begin{theo} \label{theo:Clis}
Soient $U$ une représentation lisse localement admissible de $T(F)$ sur $A$, $w \in W$ et $n \in \Nbb$. On suppose ou bien $A=\ke$ et $U$ de dimension $1$, ou bien $n \leq 1$. Alors on a un isomorphisme $T(F)$-équivariant
\begin{equation*}
\HOrd[n](\Clis_c(N_w(F),U)) \cong
\begin{cases}
U^w \otimes (\oma_w) &\text{si $n = [F:\Qp] \cdot \ell(w)$,} \\
0& \text{sinon.}
\end{cases}
\end{equation*}
\end{theo}

\begin{proof}
On reprend les notations de la sous-section \ref{sec:calc}. En utilisant les lemmes \ref{lemm:V0} et \ref{lemm:filx} et la remarque \ref{rema:xab}, on voit que l'on a une suite exacte courte dans la catégorie $\mathcal{X}$ du lemme \ref{lemm:x}
\begin{equation*}
0 \to \Hc[n](N_0,\Vwz) \to \Hc[n](N_0,V_w) \to \Hc[n](N_0,V_w/\Vwz) \to 0
\end{equation*}
et l'action de Hecke de $\varsigma$ sur $\Hc[n](N_0,V_w/\Vwz)$ est localement nilpotente. En appliquant le foncteur $\HomTp{-}$ on obtient donc un isomorphisme $T(F)$-équivariant
\begin{equation*}
\HOrd[n](V_w) \cong \HomTp{\Hc[n](N_0,\Vwz)}.
\end{equation*}

Si $n = [F:\Qp] \cdot \ell(w)$, alors on a un isomorphisme $T^+$-équivariant $\Hc[n](N_0,\Vwz) \cong U^w \otimes (\oma_0)$ d'après le lemme \ref{lemm:H2} avec $k=0$, d'où (en remarquant que $\alpha_w=\alpha_0$), un isomorphisme $T(F)$-équivariant
\begin{equation*}
\HomTp{\Hc[n](N_0,\Vwz)} \cong U^w \otimes (\oma_w).
\end{equation*}

Si $n \neq [F:\Qp] \cdot \ell(w)$, alors ou bien $n < \ell(w)$, donc l'action de Hecke de $\varsigma$ est nilpotente sur $\Hc[n](N_0,\Vwz)$ d'après le lemme \ref{lemm:H3} avec $k=0$, ou bien $n>\ell(w)$, donc $\Hc[n](N_0,\Vwz)=0$ d'après le lemme \ref{lemm:H1} avec $k=0$. Dans les deux cas, on en déduit que
\begin{equation*}
\HomTp{\Hc[n](N_0,\Vwz)} = 0. \qedhere
\end{equation*}
\end{proof}

\begin{rema}
Le théorème est en fait vrai en toute généralité (nous le démontrerons dans un prochain article). Cependant, l'argument qui montre que l'action est triviale dans l'itération de la récurrence du lemme \ref{lemm:H2} ne se généralise pas. Cela ne nous limitera pas dans les applications.
\end{rema}

\subsubsection*{Calcul sur une induite}

On commence par calculer le $\delta$-foncteur $\HOrd$ en degré $1$ sur une induite localement admissible.

\begin{coro} \label{coro:H1Ord}
Soit $U$ une représentation lisse localement admissible de $T(F)$ sur $A$.
\begin{enumerate}[(i)]
\item Si $F=\Qp$, alors on a un isomorphisme naturel $T(\Qp)$-équivariant
\begin{equation*}
\HOrdQp\left(\IndQp U\right) \cong \bigoplus_{\alpha \in \Delta} U^\alpha \otimes (\oma).
\end{equation*}
\item Si $F \neq \Qp$, alors $\HOrd[1] \left( \Ind U \right)=0$.
\end{enumerate}
\end{coro}

\begin{proof}
En utilisant les suites exactes \eqref{filhord} pour tout $r \in \llbrack 0,d \rrbrack$ et le théorème \ref{theo:Clis} avec $n=1$ (en remarquant que $\alpha_{s_\alpha}=\alpha$ pour tout $\alpha \in \Delta$), on obtient directement les isomorphismes des points (i) et (ii). On montre la naturalité du premier en rappelant les différentes étapes de sa construction. La filtration de Bruhat $(I_r)_{r \in \llbrack -1,d \rrbrack}$ de $\Ind U$ est fonctorielle en $U$, tout comme l'isomorphisme $B(F)$-équivariant
\begin{equation*}
I_r/I_{r-1} \cong  \bigoplus_{\ell(w)=r} \Clis_c(N_w(F),U)
\end{equation*}
pour tout $r \in \llbrack 0,d \rrbrack$. Ainsi les suites exactes \eqref{filhord} sont fonctorielles en $U$ pour tout $r \in \llbrack 0,d \rrbrack$. On suppose $F=\Qp$. En tenant compte des cas d'annulation du théorème \ref{theo:Clis} avec $n=1$, on en déduit un isomorphisme naturel $T(\Qp)$-équivariant
\begin{equation*}
\HOrdQp \left(\IndQp U\right) \cong \bigoplus_{\alpha \in \Delta} \HOrdQp(\Clis_c(N_{s_\alpha}(F),U)).
\end{equation*}
Soit $\alpha \in \Delta$. Avec les notations de la sous-section \ref{sec:calc}, l'inclusion $V_{s_\alpha,0} \subset V_{s_\alpha}$ est fonctorielle en $U$. En utilisant le lemme \ref{lemm:V0} avec $w=s_\alpha$ et le lemme \ref{lemm:x}, on voit qu'elle induit un isomorphisme naturel $T(\Qp)$-équivariant
\begin{equation*}
\HOrdQp(V_{s_\alpha}) \osi A[T] \otimes_{A[T^+]} \Hc[1](N_0,V_{s_\alpha,0}).
\end{equation*}
Enfin, en notant $N_\alpha$ le sous-groupe radiciel de $N$ correspondant à $\alpha$ et $N_{\alpha,0}$ l'intersection de $N_\alpha(\Qp)$ avec $N_0$, on a des isomorphismes naturels $T^+$-équivariants
\begin{align*}
\Hc[1](N_0,V_{s_\alpha,0}) &\osi \Hc[1](N_{\alpha,0},V_{s_\alpha,0}^{N_{s_\alpha,0}}) \\
& \iso \Hc[1](N_{\alpha,0},U^\alpha) \\
& \iso U^\alpha \otimes (\oma).
\end{align*}
Le premier est l'inflation, le second est induit par l'évaluation en $1 \in N_{s_\alpha}(\Qp)$ et le dernier est l'isomorphisme de la proposition \ref{prop:dimN}.
\end{proof}

\begin{defi}
Un \og \emph{twisting element} \fg\footnote{Terminologie d'après \cite{BG}.} de $G$ est un caractère algébrique $\theta$ de $T$ tel que $\langle\theta,\alpha^\vee\rangle=1$ pour tout $\alpha \in \Delta$.
\end{defi}

\begin{rema}
Lorsque le groupe dérivé de $G$ est simplement connexe, la somme des poids fondamentaux relatifs à $\Delta$ est un \og twisting element \fg{} bien défini à un caractère algébrique de $G$ près.
\end{rema}

On calcule à présent le $\delta$-foncteur $\HOrd$ en degré quelconque sur l'induite d'un caractère lorsque $G$ admet un \og twisting element \fg{} que l'on note $\theta$.

\begin{coro} \label{coro:HnOrd}
Soient $\chi : T(F) \to A^\times$ un caractère lisse et $n \in \Nbb$. On suppose $A=\ke$ ou $n \leq 1$. Alors on a un isomorphisme $T(F)$-équivariant
\begin{equation*}
\HOrd[n]\left(\Ind \chi \omth\right) \cong \bigoplus_{[F:\Qp] \cdot \ell(w)=n} w(\chi) \omth.
\end{equation*}
\end{coro}

\begin{proof}
En utilisant les suites exactes \eqref{filhord} pour tout $r \in \llbrack 0,d \rrbrack$ et le théorème \ref{theo:Clis} avec $U = \chi \omth$, on obtient un isomorphisme $T(F)$-équivariant
\begin{equation*}
\HOrd[n]\left(\Ind \chi \omth\right) 
\cong \bigoplus_{[F:\Qp] \cdot \ell(w)=n} \left( \chi \cdot (\omega^{-1} \circ \theta) \right)^w \otimes (\oma_w)
\end{equation*}
et pour tout $w \in W$, on a par définition
\begin{equation*}
\left( \chi \cdot (\omega^{-1} \circ \theta) \right)^w \otimes (\oma_w) = w^{-1}(\chi) \cdot (\omega^{-1} \circ (w^{-1}(\theta) + \alpha_w)).
\end{equation*}
En faisant le changement de variable $w \mapsto w^{-1}$ dans la somme directe (possible car $\ell(w)=\ell(w^{-1})$), on voit qu'il suffit de montrer que $w^{-1}(\theta) + \alpha_w = \theta$ pour tout $w \in W$.

Soit $\rho=\frac12\sum_{\alpha \in \Phi^+} \alpha$ la demi-somme des racines positives (qui n'est pas un caractère algébrique de $T$ en général). Pour tout $w \in W$, on a
\begin{equation*}
w^{-1}(\rho)+\alpha_w = \left( \frac12 \sum_{\substack{\alpha \in \Phi^+ \\ w(\alpha) \in \Phi^+}} \alpha - \frac12 \sum_{\substack{\alpha \in \Phi^+ \\ w(\alpha) \notin \Phi^+}} \alpha \right) + \sum_{\substack{\alpha \in \Phi^+ \\ w(\alpha) \notin \Phi^+}} \alpha = \rho.
\end{equation*}
D'après \cite[Partie II, § 1.5]{Jan}, on a $\langle \rho,\alpha^\vee \rangle = 1$, donc $\langle\theta-\rho,\alpha^\vee\rangle=0$ pour tout $\alpha \in \Delta$. On en déduit que $\theta-\rho$ est invariant sous l'action de $W$, d'où
\begin{equation*}
w^{-1}(\theta)+\alpha_w = w^{-1}(\theta-\rho) + (w^{-1}(\rho) + \alpha_w) = (\theta-\rho) + \rho = \theta
\end{equation*}
pour tout $w \in W$.
\end{proof}

\subsection{Variante pour les induites paraboliques}

Soient $P \subset G$ un sous-groupe parabolique standard et $L \subset P$ le sous-groupe de Levi standard. On reprend les notations de la sous-section \ref{sec:para}. On calcule le $\delta$-foncteur $\HOrd$ sur les induites paraboliques d'un caractère.

\subsubsection*{Calculs en degré quelconque}

Soient $\eta :F^\times \to A^\times$ un caractère lisse et $\detfr$ un caractère algébrique de $G$.
Le lemme \ref{lemm:filx} et sa démonstration sont vrais tels quels avec $P^-$, $L$ et $\eta \circ \detfr$ au lieu de $B^-$, $T$ et $U$ et les suites exactes \eqref{filindP} et \eqref{filcohomP} au lieu de \eqref{filind} et \eqref{filcohom}. En utilisant le lemme \ref{lemm:x}, on voit que pour tout $r \in \llbrack 0,d \rrbrack$ on a des suites exactes courtes de représentations lisses de $T(F)$ sur $A$
\begin{multline} \label{filhordP}
0 \to \HOrd (I^P_{r-1}) \to \HOrd (I^P_r) \\
\to \bigoplus_{\ell(\wP)=r} \HOrd (\Clis_c(N_{\wP}(F),\eta \circ \detfr)) \to 0.
\end{multline}
Soit $n \in \Nbb$. On suppose $A=\ke$ ou $n \leq 1$. En utilisant les suites exactes \eqref{filhordP} pour tout $r \in \llbrack 0,d \rrbrack$ et le théorème \ref{theo:Clis} avec $U=\eta \circ \detfr$, on obtient un isomorphisme $T(F)$-équivariant
\begin{equation} \label{hordeta}
\HOrd[n] \left( \IndP \eta \circ \detfr \right) \cong \bigoplus_{[F:\Qp] \cdot \ell(\wP) = n} (\eta \circ \detfr) \cdot (\oma_{\wP}).
\end{equation}

\subsubsection*{Calculs en degrés $0$ et $1$}

On suppose $P \neq B$ et on fixe un caractère algébrique de $G$ noté $\detfr$.

\begin{coro} \label{coro:eta0}
Soit $\eta : F^\times \to A^\times$ un caractère lisse. Alors
\begin{equation*}
\Ord \left( \IndP \eta \circ \detfr \right) = 0.
\end{equation*}
\end{coro}

\begin{proof}
Soit $\wP \in \WP$. Par définition de $\WP$, on a $\ell(\wP) \geq \ell(w_{L,0})$ avec égalité si et seulement si $\wP = w_{L,0}$. Comme $P \neq B$, on a $w_{L,0} \neq 1$ donc $\ell(\wP)>0$, d'où $[F: \Qp] \cdot \ell(\wP)>0$. On en déduit que la somme directe de l'isomorphisme \eqref{hordeta} avec $n=0$ est nulle.
\end{proof}

Pour tout $\alpha \in \Delta$, on note $P_\alpha$ le sous-groupe parabolique standard de $G$ correspondant.

\begin{coro} \label{coro:eta1}
Soit $\eta : F^\times \to A^\times$ un caractère lisse. Alors
\begin{equation*}
\HOrd[1] \left( \IndP \eta \circ \detfr \right) = 0
\end{equation*}
sauf si $F=\Qp$ et $P=P_\alpha$ avec $\alpha \in \Delta$, auquel cas on a un isomorphisme $T(\Qp)$-équivariant
\begin{equation*}
\HOrdQp \left( \IndPa \eta \circ \detfr \right) \cong (\eta \circ \detfr) \cdot (\oma).
\end{equation*}
\end{coro}

\begin{proof}
Soit $\wP \in \WP$. Si $[F:\Qp] \cdot \ell(\wP)=1$, alors $[F:\Qp]=\ell(\wP)=1$ donc $F=\Qp$ et $\wP = w_{L,0} = s_\alpha$ avec $\alpha \in \Delta$, d'où $P=P_\alpha$. On en déduit que l'isomorphisme \eqref{hordeta} avec $n=1$ est non nul si et seulement si $F=\Qp$ et $P=P_\alpha$ avec $\alpha \in \Delta$, auquel cas on obtient bien l'isomorphisme de l'énoncé (en utilisant le fait que $\alpha_{s_\alpha}=\alpha$).
\end{proof}

\section{Application aux extensions} \label{sec:5}

Nous étudions les caractères de $T(F)$ et leurs extensions. Puis, nous calculons les $\Ext^1$ entre certaines induites de $G(F)$ dans la catégories des représentations continues unitaires admissibles de $G(F)$ sur $E$. Nos démonstrations prouvent également les résultats analogues dans la catégorie des représentations lisses admissibles de $G(F)$ sur $\ke$. Enfin, en supposant vraie la conjecture \ref{conj:ROrd}, nous calculons les $\Ext^\bullet$ entre certaines induites de $G(F)$ dans la catégorie des représentations lisses localement admissibles de $G(F)$ sur $\ke$.

\subsection{\texorpdfstring{Caractères de $T(F)$ et extensions}{Caractères de T(F) et extensions}}

On définit et on étudie des notions de généricité pour les caractères de $T(F)$. On détermine ensuite les extensions entre ces derniers.

\subsubsection*{Caractères génériques}

Soit $\chi$ un caractère de $T(F)$ à valeurs dans le groupe des unités d'un anneau quelconque. Si $w \in W$, alors on définit un caractère $w(\chi)$ de $T(F)$ par $w(\chi)(t) = \chi(\w^{-1} t \w)$ pour tout $t \in T(F)$.

\begin{defi}
On dit que $\chi$ est :
\begin{itemize}
\item \emph{faiblement générique} si $s_\alpha(\chi) \neq \chi$ pour tout $\alpha \in \Delta$ ;
\item \emph{fortement générique} si $w(\chi) \neq \chi$ pour tout $w \in W -\{1\}$.
\end{itemize}
\end{defi}

On étudie ces notions de généricité dans les lemmes ci-dessous. En particulier, on les compare à \cite[Définition 3.3.1]{BH}.

\begin{lemm} \label{lemm:gen1}
Soit $\alpha \in \Phi^+$.
\begin{enumerate}[(i)]
\item Si $s_\alpha(\chi) \neq \chi$, alors $\chi \circ \alpha^\vee \neq 1$.
\item Si le centre de $G$ est connexe, alors la réciproque est vraie.
\end{enumerate}
\end{lemm}

\begin{proof}
Le lemme ne dépendant pas de $B$, on peut supposer $\alpha \in \Delta$.
Si $\chi \circ \alpha^\vee = 1$, alors pour tout cocaractère algébrique $\lambda$ de $T$ on a
\begin{equation*}
s_\alpha(\chi) \circ \lambda = \chi \circ s_\alpha(\lambda) = \chi \circ (\lambda - \langle \alpha,\lambda \rangle \alpha^\vee) = (\chi \circ \lambda) \cdot (\chi \circ \alpha^\vee)^{-\langle \alpha,\lambda \rangle} =\chi \circ \lambda.
\end{equation*}
Comme les images des cocaractères algébriques de $T$ engendrent $T(F)$, on en déduit le point (i).
Réciproquement si $s_\alpha(\chi)=\chi$, alors pour tout cocaractère algébrique $\lambda$ de $T$ on a
\begin{equation*}
\chi \circ (\lambda-s_\alpha(\lambda)) = (\chi \circ \lambda) \cdot (\chi \circ s_\alpha(\lambda))^{-1} = (\chi \circ \lambda) \cdot (s_\alpha(\chi) \circ \lambda)^{-1}=1.
\end{equation*}
Si le centre de $G$ est connexe, alors avec $\lambda$ le copoids fondamental relatif à $\alpha$ on a $\lambda-s_\alpha(\lambda)=\alpha^\vee$, d'où le point (ii).
\end{proof}

\begin{lemm} \label{lemm:gen2}
Soit $\alpha \in \Delta$. On suppose le centre de $G$ connexe et $s_\alpha(\chi) \neq \chi$. Alors pour tout $\beta \in \Delta$, on a $s_\alpha(\chi) = s_\beta(\chi)$ si et seulement si $\alpha=\beta$.
\end{lemm}

\begin{proof}
Soit $\beta \in \Delta$. Pour tout cocaractère algébrique $\lambda$ de $T$, on a
\begin{equation*}
(\chi \circ \lambda) \cdot (s_\beta s_\alpha(\chi) \circ \lambda)^{-1} = \chi \circ (\lambda-s_\alpha s_\beta(\lambda)) = \chi \circ ((\lambda-s_\alpha(\lambda))+s_\alpha(\lambda-s_\beta(\lambda))).
\end{equation*}
Avec $\lambda$ le copoids fondamental relatif à $\alpha$, on a $\lambda-s_\alpha(\lambda)=\alpha^\vee$ et si $\beta \neq \alpha$, alors $\lambda-s_\beta(\lambda)=0$, d'où
\begin{equation*}
\chi \circ ((\lambda-s_\alpha(\lambda))+s_\alpha(\lambda-s_\beta(\lambda))) = \chi \circ \alpha^\vee \neq 1
\end{equation*}
d'après le lemme \ref{lemm:gen1} et on en déduit que $s_\beta s_\alpha(\chi) \neq \chi$.
\end{proof}

\subsubsection*{Extensions entre caractères}

On calcule les $\Ext^\bullet$ entre caractères lisses de $T(F)$ dans la catégorie des représentations lisses localement admissibles de $T(F)$ sur $\ke$ en utilisant une résolution injective. En degré $1$, il revient au même de calculer les $\Ext^1$ dans la catégorie des représentations lisses admissibles de $T(F)$ sur $\ke$ en utilisant les extensions de Yoneda (voir \cite[Proposition 2.2.13]{Em1}).

\begin{prop} \label{prop:extTke}
Soit $\chi : T(F) \to \ke^\times$ un caractère lisse.
\begin{enumerate}[(i)]
\item Soit $\chi' : T(F) \to \ke^\times$ un autre caractère lisse. Si $\chi' \neq \chi$, alors
\begin{equation*}
\Ext_{T(F)}^\bullet(\chi',\chi)=0.
\end{equation*}
\item Si $F$ ne contient pas de racine primitive $p$-ième de l'unité, alors
\begin{equation*}
\dim_{\ke} \Ext_{T(F)}^1(\chi,\chi) = \left( [F:\Qp] + 1 \right) \rg(G).
\end{equation*}
\item Si $F$ contient une racine primitive $p$-ième de l'unité, alors
\begin{equation*}
\dim_{\ke} \Ext_{T(F)}^1(\chi,\chi) = \left( [F:\Qp] + 2 \right) \rg(G).
\end{equation*}
\end{enumerate}
\end{prop}

\begin{proof}
Le point (i) est une généralisation immédiate de \cite[Lemme 4.3.10]{Em2}. Pour les points (ii) et (iii), on utilise les isomorphismes $\ke$-linéaires suivants
\begin{equation*}
\Ext_{T(F)}^1 \left( \chi,\chi \right) \cong \Ext_{T(F)}^1 \left( \un,\un \right) \cong \HomGr(T(F),\ke)
\end{equation*}
avec $\un$ la représentation triviale de $T(F)$ sur $\ke$ et $\HomGr$ les morphismes de groupes continus. On a des isomorphismes de groupes topologiques
\begin{equation*}
T(F) \cong \left( F^\times \right)^{\rg(G)} \quad \text{et} \quad F^\times \cong \Z \times \left( \Z/q\Z \right)^\times \times \left( \Z/p^r\Z \right) \times \Zp^{[F:\Qp]}
\end{equation*}
avec $q$ le cardinal du corps résiduel de $F$ et $r$ le plus grand entier naturel tel que $F$ contienne une racine primitive $p^r$-ième de l'unité. Les $\ke$-espaces vectoriels $\HomGr(\Z,\ke)$ et $\HomGr(\Zp,\ke)$ sont de dimension $1$. La dimension du $\ke$-espace vectoriel $\HomGr(\Z/p^r\Z,\ke)$ est $1$ si $r>0$ et $0$ sinon. Enfin $(\Z/q\Z)^\times$ est d'ordre premier à $p$, donc $\HomGr((\Z/q\Z)^\times,\ke)=0$. On obtient ainsi les points (ii) et (iii).
\end{proof}

\begin{rema}
Pour tout entier $n>1$, on ne sait pas si l'application canonique
\begin{equation*}
\Ext_{T(F)}^n \left( \un,\un \right) \to \Hc[n] \left( T(F),\un \right)
\end{equation*}
est un isomorphisme car $T(F)$ n'est pas compact (voir \cite[§ 2.2]{Em2}).
\end{rema}

On calcule les $\Ext^1$ entre caractères continus unitaires de $T(F)$ dans la catégorie des représentations continues unitaires admissibles de $T(F)$ sur $E$ en utilisant les extensions de Yoneda.

\begin{prop} \label{prop:extTE}
Soient $\chi,\chi' : T(F) \to \Oe^\times \subset E^\times$ des caractères continus unitaires. Alors
\begin{equation*}
\Ext_{T(F)}^1(\chi',\chi)=0
\end{equation*}
sauf si $\chi' = \chi$, auquel cas
\begin{equation*}
\dim_E \Ext_{T(F)}^1(\chi,\chi) = \left( [F:\Qp] + 1 \right) \rg(G).
\end{equation*}
\end{prop}

\begin{proof}
Si $\chi' \neq \chi$, alors il existe $t \in T(F)$ tel que $\chi'(t) \neq \chi(t)$, donc pour toute extension $\mathcal{E}$ de $\chi'$ par $\chi$, l'endomorphisme $t$ de $\mathcal{E}$ est diagonalisable (car son polynôme minimal est scindé à racines simples) et comme il commute avec $T(F)$, on obtient ainsi une section de l'extension $\mathcal{E}$. Si $\chi'=\chi$, alors la dimension se calcule comme dans la preuve des points (ii) et (iii) de la proposition \ref{prop:extTke}, sauf que cette fois $\HomGr(\Z/p^r\Z,E)=0$ quel que soit l'entier naturel $r$ puisque $E$ est sans torsion.
\end{proof}

\subsection{Extensions entre induites}

On détermine les extensions entre certaines induites de $G(F)$. On calcule les $\Ext^1$ dans la catégorie des représentations continues unitaires admissibles de $G(F)$ sur $E$ en utilisant les extensions de Yoneda. On démontre également les résultats analogues dans la catégorie des représentations lisses admissibles de $G(F)$ sur $\ke$.

\subsubsection*{Résultats dans le cas $F=\Qp$}

On suppose $F=\Qp$ et on fait les hypothèses suivantes\footnote{Dans un prochain article nous traiterons le cas général.} sur $G$ : son centre est connexe et son groupe dérivé est simplement connexe. On note $\theta$ la somme des poids fondamentaux relatifs à $\Delta$.

\begin{theo} \label{theo:ext1}
Soit $\chi : T(\Qp) \to  \Oe^\times \subset E^\times$ un caractère continu unitaire.
\begin{enumerate}[(i)]
\item Si $\chi' : T(\Qp) \to  \Oe^\times \subset E^\times$ est un autre caractère continu unitaire, alors
\begin{equation*}
\Ext^1_{G(\Qp)} \left( \IndQp \chi' \epsth,\IndQp \chi \epsth \right) \neq 0
\end{equation*}
si et seulement si $\chi'=\chi$ ou $\chi'=s_\alpha(\chi)$ avec $\alpha \in \Delta$.
\item Soit $\alpha \in \Delta$. Si $s_\alpha(\chi) \neq \chi$, alors
\begin{equation*}
\dim_E \Ext^1_{G(\Qp)} \left( \IndQp s_\alpha(\chi) \epsth,\IndQp \chi \epsth \right) = 1.
\end{equation*}
\item Si $\chi$ est faiblement générique, alors le foncteur $\IndQp$ induit un isomorphisme $E$-linéaire
\begin{multline*}
\Ext_{T(\Qp)}^1 \left( \chi \epsth, \chi \epsth \right) \\
\iso \Ext_{G(\Qp)}^1 \left( \IndQp  \chi \epsth,\IndQp \chi \epsth \right).
\end{multline*}
\end{enumerate}
\end{theo}

\begin{rema}
On s'attend à ce que le point (iii) soit vrai sans l'hypothèse de généricité. De même pour son analogue modulo $p$ sauf lorsque $p=2$, auquel cas on s'attend à ce que l'injection $\ke$-linéaire induite par le foncteur $\IndQp$ ait un conoyau de dimension
\begin{equation*}
\card \left\{\alpha \in \Delta ~|~ s_\alpha(\chi)=\chi\right\}.
\end{equation*}
\end{rema}

\begin{proof}
Soient $\chi,\chi' : T(\Qp) \to A^\times$ des caractères lisses. En utilisant le corollaire \ref{coro:HnOrd} avec $n=1$, la suite exacte \eqref{GroSE} avec $U = \chi' \omth$ et $V = \IndQp \chi \omth$ devient une suite exacte de $A$-modules
\begin{multline} \label{SEext}
0 \to \Ext_{T(\Qp)}^1 \left( \chi' \omth,\chi \omth \right) \\
\to \Ext_{G(\Qp)}^1\left( \IndQp \chi' \omth,\IndQp \chi \omth \right) \\
\to \bigoplus_{\beta \in \Delta} \Hom_{T(\Qp)} \left( \chi' \omth,s_\beta(\chi) \omth \right).
\end{multline}

On suppose $A=\ke$ et on prouve la version modulo $p$ du théorème. En utilisant la proposition \ref{prop:extTke}, on obtient les cas d'annulation du point (i) modulo $p$. Le premier cas de non annulation ($\chi'=\chi$) se déduit encore de la proposition \ref{prop:extTke}. On suppose $\chi' \neq \chi$ et on montre le second cas de non annulation ($\chi'=s_\alpha(\chi)$ avec $\alpha \in \Delta$). Dans ce cas, on construit comme dans \cite[§ 3.4]{BH} une telle extension non scindée par induction parabolique à partir de l'unique extension non scindée dans
\begin{equation*}
\Ext^1_{\GL_2(\Qp)} \left( \Indd s_\alpha(\chi) \omth,\Indd \chi \omth \right)
\end{equation*}
avec $B_2^-$ le sous-groupe des matrices triangulaires inférieures et $\alpha$ la racine de $\GL_2$ négative par rapport à $B_2^-$ (voir \cite[Proposition 4.3.15]{Em2}), d'où le point (i) modulo $p$.
Soit $\alpha \in \Delta$. On suppose $s_\alpha(\chi) \neq \chi$, donc $s_\alpha(\chi) \neq s_\beta(\chi)$ pour tout $\beta \in \Delta-\{\alpha\}$ d'après le lemme \ref{lemm:gen2}. En utilisant la proposition \ref{prop:extTke}, on voit que la suite exacte \eqref{SEext} avec $\chi'=s_\alpha(\chi)$ donne une injection $\ke$-linéaire
\begin{multline*}
\Ext_{G(\Qp)}^1 \left( \IndQp s_\alpha(\chi) \omth,\IndQp \chi \omth \right) \\
\hookrightarrow \Hom_{T(\Qp)} \left( s_\alpha(\chi) \omth,s_\alpha(\chi) \omth \right).
\end{multline*}
On en déduit que la dimension de la source est au plus $1$. Comme elle est non nulle d'après le point (i) modulo $p$, on en déduit que sa dimension est exactement $1$, d'où le point (ii) modulo $p$.
Enfin si $\chi$ est faiblement générique, alors $s_\beta(\chi) \neq \chi$ pour tout $\beta \in \Delta$, donc la suite exacte \eqref{SEext} avec $\chi'=\chi$ donne un isomorphisme $\ke$-linéaire
\begin{multline*}
\Ext_{T(\Qp)}^1 \left( \chi \omth,\chi \omth \right) \\
\iso \Ext_{G(\Qp)}^1 \left( \IndQp \chi \omth,\IndQp \chi \omth \right),
\end{multline*}
d'où le point (iii) modulo $p$.

On prouve maintenant le théorème. Soient $\chi,\chi' : T(\Qp) \to \Oe^\times \subset E^\times$ des caractères continus unitaires. Pour $k \geq 1$ entier, les suites exactes \eqref{SEext} avec $A=\A{k}$ et les réductions modulo $\pe^k$ des caractères $\chi$ et $\chi'$ forment un système projectif. En passant à la limite projective puis en tensorisant par $E$ sur $\Oe$, on obtient en utilisant \cite[Lemme 4.1.3]{Em1}, l'isomorphisme \eqref{homprojlim} et la proposition \ref{prop:extprojlim} une suite exacte de $E$-espaces vectoriels
\begin{multline*}
0 \to \Ext_{T(\Qp)}^1 \left( \chi' \epsth,\chi \epsth \right) \\
\to \Ext_{G(\Qp)}^1 \left( \IndQp \chi' \epsth,\IndQp \chi \epsth \right) \\
\to \bigoplus_{\beta \in \Delta} \Hom_{T(\Qp)} \left( \chi' \epsth,s_\beta(\chi) \epsth \right).
\end{multline*}
On utilise la proposition \ref{prop:extTE} pour calculer les extensions entre caractères continus unitaires de $T(\Qp)$ et on montre que si $\alpha \in \Delta$ et $s_\alpha(\chi) \neq \chi$, alors
\begin{equation*}
\Ext_{G(\Qp)}^1\left( \IndQp s_\alpha(\chi) \epsth,\IndQp \chi \epsth \right) \neq 0
\end{equation*}
en construisant comme dans \cite[§ 3.3]{BH} une telle extension non scindée par induction parabolique à partir de l'unique extension non scindée dans
\begin{equation*}
\Ext^1_{\GL_2(\Qp)} \left( \Indd s_\alpha(\chi) \epsth,\Indd \chi \epsth \right)
\end{equation*}
(ce $E$-espace vectoriel est de dimension $1$ d'après \cite[Proposition B.2]{BH}). Le reste de la démonstration est identique à la version modulo $p$. 
\end{proof}

Soient $P \subset G$ un sous-groupe parabolique standard et $\detfr$ un caractère algébrique de $G$. On suppose $P \neq B$.

\begin{prop}
Soient $\chi : T(\Qp) \to \Oe^\times \subset E^\times$ et $\eta :\Qp^\times \to \Oe^\times \subset E^\times$ des caractères continus unitaires. Alors
\begin{equation*}
\Ext^1_{G(\Qp)} \left( \IndQp \chi , \IndPQp \eta \circ \detfr \right) = 0
\end{equation*}
sauf si $P=P_\alpha$ avec $\alpha \in \Delta$ et $\chi = (\eta \circ \detfr) \cdot (\epsa)$, auquel cas
\begin{equation*}
\dim_E \Ext^1_{G(\Qp)} \left( \IndQp (\eta \circ \detfr) \cdot (\epsa) , \IndPa \eta \circ \detfr \right) = 1.
\end{equation*}
\end{prop}

\begin{proof}
Soient $\chi : T(\Qp) \to A^\times$ et $\eta : \Qp^\times \to A^\times$ des caractères lisses. En utilisant le corollaire \ref{coro:eta0}, la suite exacte \eqref{GroSE} avec $U = \chi$ et $V = \IndPQp \eta \circ \detfr$ donne une injection $A$-linéaire
\begin{multline*}
\Ext_{G(\Qp)}^1\left( \IndQp \chi,\IndPQp \eta \circ \detfr \right) \\
\hookrightarrow \Hom_{T(\Qp)} \left( \chi, \HOrdQp \left( \IndPQp \eta \circ \detfr \right) \right).
\end{multline*}
En utilisant le corollaire \ref{coro:eta1}, on en déduit que
\begin{equation} \label{SEextP1}
\Ext^1_{G(\Qp)} \left( \IndQp \chi , \IndPQp \eta \circ \detfr \right) = 0
\end{equation}
sauf si $P=P_\alpha$ avec $\alpha \in \Delta$, auquel cas on a une injection $A$-linéaire
\begin{multline} \label{SEextP2}
\Ext_{G(\Qp)}^1\left( \IndQp \chi,\IndPa \eta \circ \detfr \right) \\ \hookrightarrow \Hom_{T(\Qp)} \left( \chi,(\eta \circ \detfr) \cdot (\oma) \right).
\end{multline}

On suppose $A=\ke$ et on prouve la version modulo $p$ de la proposition. On déduit de \eqref{SEextP1} et \eqref{SEextP2} les cas d'annulation et on voit que si $\alpha \in \Delta$, alors
\begin{equation*}
\Ext^1_{G(\Qp)} \left( \IndQp (\eta \circ \detfr) \cdot (\oma) , \IndPa \eta \circ \detfr \right)
\end{equation*}
est de dimension au plus $1$. On montre qu'il est non nul en construisant comme dans \cite[§ 3.4]{BH} une telle extension non scindée par induction parabolique à partir de l'unique extension non scindée dans
\begin{equation*}
\Ext^1_{\GL_2(\Qp)} \left( \Indd (\eta \circ \detfr) \cdot (\oma) ,\eta \circ \detfr \right)
\end{equation*}
avec $B_2^-$ le sous-groupe des matrices triangulaires inférieures et $\alpha$ la racine de $\GL_2$ négative par rapport à $B_2^-$ (voir \cite[Proposition 4.3.13]{Em2}). Ainsi la dimension est exactement $1$, d'où la proposition modulo $p$.

On prouve maintenant la proposition. Soient $\chi : T(\Qp) \to \Oe^\times \subset E^\times$ et $\eta : \Qp^\times \to \Oe^\times \subset E^\times$ des caractères continus unitaires. Pour $k \geq 1$ entier, \eqref{SEextP1} et \eqref{SEextP2} avec $A=\A{k}$ et les réductions modulo $\pe^k$ des caractères $\chi$ et $\eta$ forment des systèmes projectifs. On passe à la limite projective puis on tensorise par $E$ sur $\Oe$ comme dans la preuve du théorème \ref{theo:ext1}. On montre que si $\alpha \in \Delta$, alors
\begin{equation*}
\Ext^1_{G(\Qp)} \left( \IndQp (\eta \circ \detfr) \cdot (\epsa) , \IndPa \eta \circ \detfr \right) \neq 0
\end{equation*}
en construisant comme dans \cite[§ 3.3]{BH} une telle une extension non scindée par induction parabolique à partir de l'unique extension non scindée dans
\begin{equation*}
\Ext^1_{\GL_2(\Qp)} \left( \Indd (\eta \circ \detfr) \cdot (\epsa) ,\eta \circ \detfr \right)
\end{equation*}
(on procède comme dans la preuve de \cite[Proposition B.2]{BH} pour montrer que ce $E$-espace vectoriel est de dimension $1$). Le reste de la démonstration est identique à la version modulo $p$.
\end{proof}

\subsubsection*{Résultats dans le cas $F \neq \Qp$}

On suppose $F \neq \Qp$ et on ne fait aucune hypothèse sur $G$.

\begin{theo} \label{theo:ext1F}
Soit $\chi : T(F) \to \Oe^\times \subset E^\times$ un caractère continu unitaire.
\begin{enumerate}[(i)]
\item Si $\chi' : T(F) \to \Oe^\times \subset E^\times$ est un autre caractère continu unitaire, alors
\begin{equation*}
\Ext_{G(F)}^1 \left( \Ind \chi', \Ind \chi \right) \neq 0
\end{equation*}
si et seulement si $\chi'=\chi$.
\item Le foncteur $\Ind$ induit un isomorphisme $E$-linéaire
\begin{equation*}
\Ext_{T(F)}^1 \left( \chi,\chi \right) \iso \Ext_{G(F)}^1 \left( \Ind \chi, \Ind \chi \right).
\end{equation*}
\end{enumerate}
\end{theo}

\begin{proof}
Soient $\chi,\chi' : T(F) \to A^\times$ des caractères lisses. En utilisant le corollaire \ref{coro:H1Ord}, la suite exacte \eqref{GroSE} avec $U = \chi'$ et $V = \Ind \chi$ donne un isomorphisme $A$-linéaire
\begin{equation*}
\Ext_{T(F)}^1 \left( \chi',\chi \right) \iso \Ext_{G(F)}^1\left( \Ind \chi',\Ind \chi \right).
\end{equation*}
Avec $A=\ke$ et en utilisant la proposition \ref{prop:extTke} on obtient la version modulo $p$ du théorème. Par un passage à la limite projective et un produit tensoriel comme dans la preuve du théorème \ref{theo:ext1} et en utilisant la proposition \ref{prop:extTE} on obtient le théorème.
\end{proof}

Soient $P \subset G$ un sous-groupe parabolique standard et $\detfr$ un caractère algébrique de $G$. On suppose $P \neq B$.

\begin{prop}
Soient $\chi : T(F) \to \Oe^\times \subset E^\times$ et $\eta : F^\times \to \Oe^\times \subset E^\times$ des caractères continus unitaires. Alors
\begin{equation*}
\Ext^1_{G(F)} \left( \Ind \chi , \IndP \eta \circ \detfr \right) = 0.
\end{equation*}
\end{prop}

\begin{proof}
Soient $\chi : T(F) \to A^\times$ et $\eta : F^\times \to A^\times$ des caractères lisses. En utilisant les corollaires \ref{coro:eta0} et \ref{coro:eta1}, on déduit de la suite exacte \eqref{GroSE} avec $U = \chi$ et $V = \IndP \eta \circ \detfr$ que
\begin{equation*}
\Ext_{G(F)}^1\left( \Ind \chi,\IndP \eta \circ \detfr \right) = 0.
\end{equation*}
Avec $A=\ke$ on obtient la version modulo $p$ de la proposition. Par un passage à la limite projective et un produit tensoriel comme dans la preuve du théorème \ref{theo:ext1} on obtient la proposition.
\end{proof}

\subsection{\texorpdfstring{Résultat conjectural sur les $\Ext^\bullet$}{Résultat conjectural sur les Ext}}

Dans cette sous-section on suppose que la conjecture \ref{conj:ROrd} est vraie (on rappelle que cette dernière est démontrée pour $\GL_2$ dans \cite{EmP}). Par soucis de clarté, on fait également l'hypothèse que $G$ admet un \og twisting element \fg{} noté $\theta$.
On calcule les $\Ext^\bullet$ entre certaines induites (admissibles) de $G(F)$ dans la catégorie des représentations lisses \emph{localement} admissibles de $G(F)$ sur $\ke$.

\begin{theo} \label{theo:extn}
Soient $\chi : T(F) \to \ke^\times$ un caractère lisse et $n \in \Nbb$.
\begin{enumerate}[(i)]
\item Soit $\chi' : T(F) \to \ke^\times$ un autre caractère lisse. Si
\begin{equation*}
\Ext_{G(F)}^n \left( \Ind \chi' \omth, \Ind \chi \omth \right) \neq 0,
\end{equation*}
alors $\chi'=w(\chi)$ avec $w \in W$ et $[F:\Qp] \cdot \ell(w) \leq n$.
\item Soit $w \in W$ tel que $[F:\Qp] \cdot \ell(w) \leq n$. Si $\chi$ est fortement générique, alors on a un isomorphisme $\ke$-linéaire
\begin{multline*}
\Ext_{G(F)}^n \left( \Ind w(\chi) \omth, \Ind \chi \omth \right) \\
\cong \Ext_{T(F)}^{n-[F:\Qp] \cdot \ell(w)} \left( \un,\un \right)
\end{multline*}
avec $\un$ la représentation triviale de $T(F)$ sur $\ke$.
\end{enumerate}
\end{theo}

\begin{rema}
Sous les conditions du point (ii), on déduit du théorème que
\begin{equation*}
\dim_{\ke} \Ext_{G(F)}^{[F:\Qp] \cdot \ell(w)} \left( \Ind w(\chi) \omth, \Ind \chi \omth \right) = 1.
\end{equation*}
Lorsque $F=\Qp$, la non annulation de cet espace était attendue par Emerton (voir \cite[Remarque 3.5.2]{BH}).
\end{rema}

\begin{proof}
On note $(\mathcal{E}^{i,j}_r)$ la suite spectrale \eqref{Gro} avec $A=\ke$, $U=\chi' \omth$ et $V=\Ind \chi \omth$. En utilisant la conjecture \ref{conj:ROrd} et le corollaire \ref{coro:HnOrd}, on voit que
\begin{multline*}
\mathcal{E}^{i,j}_2 = \bigoplus_{[F:\Qp] \cdot \ell(w')=j} \Ext_{T(F)}^i \left( \chi' \omth,w'(\chi) \omth \right) \\
\Rightarrow \Ext_{G(F)}^{i+j} \left( \Ind \chi' \omth,\Ind \chi \omth \right).
\end{multline*}

On suppose que le $\ke$-espace vectoriel du point (i) est non nul. Alors il existe $i,j \in \Nbb$ tels que $i+j=n$ et $\mathcal{E}^{i,j}_2 \neq 0$, donc il existe $w' \in W$ tel que $[F:\Qp] \cdot \ell(w')=j$ et
\begin{equation*}
\Ext_{T(F)}^i \left( \chi' \omth,w'(\chi) \omth \right) \neq 0.
\end{equation*}
En utilisant la proposition \ref{prop:extTke}, on en déduit que $\chi'=w'(\chi)$. Comme de plus $[F:\Qp] \cdot \ell(w')=j \leq n$, on obtient le point (i).

On suppose $\chi$ fortement générique. Si $\chi' = w(\chi)$ avec $w \in W$, alors pour tout $w' \in W$ on a $\chi' = w'(\chi)$ si et seulement si $w'=w$. En utilisant la proposition \ref{prop:extTke} on voit que la suite spectrale $(\mathcal{E}^{i,j}_r)$ dégénère : pour tous $i,j \in \Nbb$, on a $\mathcal{E}_2^{i,j}=0$ si $j \neq [F:\Qp] \cdot \ell(w)$) et on a
\begin{equation*}
\mathcal{E}_2^{i,[F:\Qp] \cdot \ell(w)} = \Ext_{T(F)}^i \left( w(\chi) \omth,w(\chi) \omth \right).
\end{equation*}
En comparant ceci à la limite de $(\mathcal{E}^{i,j}_r)$ on voit que si $n< [F:\Qp] \cdot \ell(w)$, alors
\begin{equation*}
\Ext_{G(F)}^n \left( \Ind w(\chi) \omth, \Ind \chi \omth \right)=0
\end{equation*}
tandis que si $n \geq [F:\Qp] \cdot \ell(w)$, alors on a un isomorphisme $\ke$-linéaire
\begin{multline*}
\Ext_{G(F)}^n \left( \Ind w(\chi) \omth, \Ind \chi \omth \right) \\
\cong \Ext_{T(F)}^{n-[F:\Qp] \cdot \ell(w)} \left( w(\chi) \omth,w(\chi) \omth \right).
\end{multline*}
Après torsion par l'inverse du caractère $w(\chi) \omth$, on obtient le point (ii).
\end{proof}

\begin{rema}
Soient $\chi : T(F) \to \ke^\times$ un caractère lisse fortement générique et $w \in W$. Si $w=w_1w_2$ avec $w_1,w_2 \in W$ et $\ell(w)=\ell(w_1)+\ell(w_2)$, alors on s'attend à ce que l'unique extension non nulle dans
\begin{equation*}
\Ext_{G(F)}^{[F:\Qp] \cdot \ell(w)} \left( \Ind w(\chi) \omth, \Ind \chi \omth \right)
\end{equation*}
soit le produit de Yoneda de celles dans
\begin{equation*}
\Ext_{G(F)}^{[F:\Qp] \cdot \ell(w_2)} \left( \Ind w_2(\chi) \omth, \Ind \chi \omth \right)
\end{equation*}
et
\begin{equation*}
\Ext_{G(F)}^{[F:\Qp] \cdot \ell(w_1)} \left( \Ind w_1 w_2(\chi) \omth, \Ind w_2(\chi) \omth \right).
\end{equation*}
\end{rema}

\appendix
\numberwithin{equation}{section}

\section{Algèbre de Lie et sous-groupes standards} \label{app:Lie}

Nous donnons quelques résultats sur l'algèbre de Lie de $N$ en lien avec certains sous-groupes ouverts compacts de $N(F)$.

\subsection*{Algèbre de Lie et exponentielle}

Soit $\Lie(N)$ l'algèbre de Lie de $N$ sur $F$ (c'est-à-dire le noyau de la surjection canonique $N(F[\epsilon]) \twoheadrightarrow N(F)$ avec $\epsilon^2=0$). C'est une $F$-algèbre de Lie nilpotente, d'indice de nilpotence inférieur ou égal à $d$. On déduit de \cite[Chapitre IV, § 2, Proposition 4.1]{DG} que l'exponentielle induit un homéomorphisme canonique
\begin{equation*}
\exp : \Lie(N) \iso N(F).
\end{equation*}
On note $\log : N(F) \iso \Lie(N)$ l'inverse de $\exp$. Si $\N$ est un sous-groupe fermé de $N$, alors l'exponentielle identifie $\N(F)$ avec $\Lie(\N) \subset \Lie(N)$. Si de plus $\N$ est distingué dans $N$, alors $\Lie(\N)$ est un idéal de $\Lie(N)$ et l'exponentielle identifie $N(F)/\N(F)$ avec $\Lie(N)/\Lie(\N)$.

Soit $Z(X,Y)$ la série de Campbell-Hausdorff. C'est une série formelle de Lie à coefficients rationnels dont on note $Z^k(X,Y)$ la composante homogène de degré $k$ pour tout $k \in \Nbb$. On a $Z^0(X,Y)=0$ et $Z^1(X,Y)=X+Y$. À travers l'exponentielle, la loi de groupe sur $N(F)$ est donnée par la formule de Campbell-Hausdorff :
\begin{equation} \label{CH}
(\exp X)(\exp Y) = \exp (Z(X,Y))
\end{equation}
pour tous $X,Y \in \Lie(N)$ (voir \cite[Chapitre IV, § 2, Proposition 4.3]{DG}). En particulier, l'exponentielle est un isomorphisme de groupes lorsque $N$ est commutatif.
Cette formule est bien définie car les composantes homogènes de $Z(X,Y)$ de degré strictement supérieur à l'indice de nilpotence de $\Lie(N)$ sont nulles sur $\Lie(N)$ (on peut utiliser la série de Campbell-Hausdorff tronquée aux termes de degré inférieur ou égal à $d$ dans la formule \eqref{CH}).

\begin{defiapp}
Soient $\N \subset N$ un sous-groupe fermé et $\Nz$ un sous-groupe ouvert compact de $\N(F)$. On dit que $\Nz$ est \emph{standard} si $\log(\Nz) \subset \Lie(\N)$ est une sous-$\Zp$-algèbre de Lie. Dans ce cas on note $\Lie(\Nz)$ cette dernière.
\end{defiapp}

\subsection*{Représentation adjointe et décomposition radicielle}

On rappelle la décomposition radicielle de $\Lie(N)$ relative à $T$. Pour tout $\alpha \in \Phi^+$, on note $N_\alpha$ le sous-groupe radiciel de $N$ correspondant. On a une décomposition $F$-linéaire
\begin{equation} \label{rad}
\Lie(N) = \bigoplus_{\alpha \in \Phi^+} \Lie(N_\alpha)
\end{equation}
et $\Lie(N_\alpha)$ est le sous-espace propre (de dimension $1$ sur $F$) de poids $\alpha$ de la représentation adjointe de $T(F)$ sur $\Lie(N)$. Par définition l'exponentielle est un entrelacement entre l'action adjointe de $T(F)$ sur $\Lie(N)$ et l'action par conjugaison de $T(F)$ sur $N(F)$.

Soit $\N \subset N$ un sous-groupe fermé stable sous l'action par conjugaison de $T$. On a une décomposition $F$-linéaire
\begin{equation*}
\Lie(\N) = \bigoplus_{\alpha \in \widetilde{\Phi}^+} \Lie(N_\alpha)
\end{equation*}
avec $\widetilde{\Phi}^+ \subset \Phi^+$. On déduit de \cite[Partie II, § 1.7]{Jan} que le produit induit un isomorphisme (quel que soit l'ordre sur $\widetilde{\Phi}^+$)
\begin{equation} \label{prodNa}
\prod_{\alpha \in \widetilde{\Phi}^+} N_\alpha \iso \N
\end{equation}

\begin{defiapp}
Soit $N_0$ un sous-groupe standard de $N(F)$. On dit que $N_0$ est \emph{compatible avec la décomposition radicielle} si on a la décomposition $\Zp$-linéaire
\begin{equation*}
\Lie(N_0) = \bigoplus_{\alpha \in \Phi^+} (\Lie(N_0) \cap \Lie(N_\alpha)).
\end{equation*}
\end{defiapp}

\subsection*{Sous-groupes standards de $N(F)$}

On étudie des sous-groupes ouverts compacts particuliers de $N(F)$.

\begin{lemmapp} \label{lemm:compat}
Pour tout $c \in \Nbb$, il existe une base de voisinages de $1 \in N(F)$ constituée de sous-groupes standards $N_0$ compatibles avec la décomposition radicielle et tels que pour tous $n_1,n_2 \in N_0$ on ait dans $\Lie(N_0)$ la relation
\begin{equation*}
\log(n_1n_2) \equiv \log n_1 + \log n_2 \pmod{p^c}.
\end{equation*}
\end{lemmapp}

\begin{proof}
On pousse un peu plus loin la preuve de \cite[Lemme 3.5.2]{Em2}.
Soit $M_0$ un $\Zp$-réseau borné de $\Lie(N)$ (c'est-à-dire un sous-$\Zp$-module de type fini qui engendre $\Lie(N)$ sur $\Qp$) vérifiant
\begin{equation*}
M_0 = \bigoplus_{\alpha \in \Phi^+} (M_0 \cap \Lie(N_\alpha)).
\end{equation*}
La famille $(p^r M_0)_{r \in \Z}$ est une base de voisinages de $0$ dans $\Lie(N)$ constituée de $\Zp$-réseaux bornés vérifiant aussi cette décomposition radicielle. Soit $r \in \Z$. Alors $\exp(p^r M_0)$ est un ouvert compact de $N(F)$. Comme $[M_0,M_0]$ est borné et $M_0$ est un réseau, il existe $a \in \Z$ tel que $[M_0,M_0] \subset p^{-a} M_0$. On a donc
\begin{equation*}
[p^rM_0,p^rM_0] \subset p^{r-a} (p^r M_0)
\end{equation*}
et on en déduit que $p^r M_0$ est une sous-$\Zp$-algèbre de Lie lorsque $r \geq a$.
On note $-b$ le minimum de la valuation $p$-adique sur les coefficients des $Z^k(X,Y)$ pour $k \in \llbrack 2,d \rrbrack$. Pour tout entier $k>1$, on a donc
\begin{equation*}
Z^k(p^r M_0,p^r M_0) \subset p^{r-a-b} (p^r M_0)
\end{equation*}
et on en déduit que $\exp(p^r M_0)$ est un sous-groupe de $N(F)$ lorsque $r \geq a+b$. De plus,
\begin{equation*}
Z(X,Y) \equiv X+Y \pmod{p^{r-a-b} (p^r M_0)}
\end{equation*}
pour tous $X,Y \in p^r M_0$. Ainsi la famille $(\exp(p^r M_0))_{r\geq a+b+c}$ vérifie les conditions de l'énoncé.
\end{proof}

\begin{propapp} \label{prop:rad}
Soit $N_0$ un sous-groupe standard de $N(F)$ compatible avec la décomposition radicielle. Si $\N_1,\N_2 \subset N$ sont deux sous-groupes fermés stables sous l'action par conjugaison de $T$, alors
\begin{equation*}
\left(\N_1(F) \cap N_0\right) \left(\N_2(F) \cap N_0\right) = \left(\N_1(F) \N_2(F)\right) \cap N_0.
\end{equation*}
\end{propapp}

\begin{proof}
On prouve l'inclusion non triviale. En utilisant l'isomorphisme \eqref{prodNa}, on voit que l'on peut supposer $\N_1 \cap \N_2 = 1$. Soit $n =n_1n_2 \in (\N_1(F) \N_2(F)) \cap N_0$ avec $n_1 \in \N_1(F)$ et $n_2 \in \N_2(F)$. On va montrer que $\log(n_1),\log(n_2) \in \Lie(N_0)$. Puisque $N_0$ est compatible avec la décomposition radicielle, il suffit de montrer que pour tout $\alpha \in \Phi^+$ on a
\begin{equation*}
\pi_\alpha(\log n_1),\pi_\alpha(\log n_2) \in \Lie(N_0)
\end{equation*}
avec $\pi_\alpha : \Lie(N) \twoheadrightarrow \Lie(N_\alpha)$ la projection relative à la décomposition radicielle \eqref{rad}.
Soit $\alpha \in \Phi^+$.
La formule de Campbell-Hausdorff \eqref{CH} nous donne
\begin{equation*}
\log n = \log n_1 + \log n_2 + \sum_{k>1} Z^k(\log n_1,\log n_2).
\end{equation*}
En appliquant $\pi_\alpha$ à cette égalité et en utilisant le fait que $\pi_\alpha(\log n_1)=0$ ou $\pi_\alpha(\log n_2)=0$ (car $\N_1 \cap \N_2=1$), on voit qu'il suffit de montrer que pour tout entier $k>1$ on a
\begin{equation*}
\pi_\alpha \left( Z^k(\log n_1,\log n_2) \right) \in \Lie(N_0).
\end{equation*}
On note $|\alpha| = \sum_{\beta \in \Delta} c_\beta$ avec $\alpha = \sum_{\beta \in \Delta} c_\beta \beta$ la décomposition de $\alpha$ dans la base $\Delta$. Étant donné que
\begin{equation*}
[\Lie(N_{\alpha_1}),\Lie(N_{\alpha_2})] \subset \Lie(N_{\alpha_1+\alpha_2})
\end{equation*}
pour tous $\alpha_1,\alpha_2 \in \Phi^+$, on voit que pour tout entier $k>1$ on a
\begin{equation*}
\pi_\alpha  \left( Z^k(\log n_1,\log n_2) \right) = \pi_\alpha \left( Z^k \left( \sum_{|\alpha_1| < |\alpha|} \pi_{\alpha_1}(\log n_1),\sum_{|\alpha_2| < |\alpha|} \pi_{\alpha_2}(\log n_2) \right) \right).
\end{equation*}
On en déduit le résultat par récurrence sur $|\alpha|$.
\end{proof}

\section{\texorpdfstring{Représentations modulo $p$ et $p$-adiques}{Représentations modulo p et p-adiques}} \label{app:rep}

Soit $H$ un groupe de Lie $p$-adique. Nous explicitons certains liens entre les représentations continues unitaires admissibles de $H$ sur $E$ et les représentations lisses admissibles de $H$ sur $\A{k}$ avec $k \geq 1$ entier.

\subsection*{Représentations continues}

On rappelle quelques définitions et résultats concernant les représentations continues de $H$ sur $E$ (voir \cite[§ 2]{Sch}).
Une telle représentation est un $E$-espace de Banach $V$ muni d'une action $E$-linéaire de $H$ continue au sens où l'application correspondante $H \times V \to V$ est continue. On dit que $V$ est \emph{admissible} si son dual continu $V'=\Homcont_E(V,E)$ est de type fini sur l'algèbre d'Iwasawa $E \otimes_{\Oe} \Oe[[H_0]]$ de $H_0$ pour un (ou de façon équivalente tout) sous-groupe ouvert compact $H_0$ de $H$. On dit que $V$ est \emph{unitaire} s'il existe une boule $V_0 \subset V$ (c'est-à-dire un réseau ouvert borné) stable sous l'action de $H$. Les représentations continues unitaires admissibles de $H$ sur $E$ et les applications continues $E$-linéaires $H$-équivariantes forment une catégorie abélienne (en fait l'unitarité n'est pas nécessaire).

Une représentation $\pe$-adiquement continue de $H$ sur $\Oe$ est un $\Oe$-module $\pe$-adiquement séparé et complet $V_0$ dont la torsion est d'exposant borné (c'est-à-dire que le sous-$\Oe$-module de torsion est annulé par une puissance de $\pe$) muni d'une action $\Oe$-linéaire de $H$ continue pour la topologie $\pe$-adique. Dans ce cas, $V_0/\pe^k V_0$ est une représentation lisse de $H$ sur $\A{k}$ pour tout entier $k\geq1$ et $E \otimes_{\Oe} V_0$ est une représentation continue unitaire de $H$ sur $E$. On dit que $V_0$ est admissible si $V_0/\pe V_0$ est admissible, auquel cas $V_0/\pe^k V_0$ est admissible pour tout entier $k\geq1$ et $E \otimes_{\Oe} V_0$ est admissible. Réciproquement, si $V$ est une représentation continue unitaire de $H$ sur $E$, alors toute boule $V_0 \subset V$ stable par $H$ est une représentation $\pe$-adiquement continue de $H$ sur $\Oe$ et $V$ est admissible si et seulement si $V_0$ est admissible. Les représentations $\pe$-adiquement continues admissibles de $H$ sur $\Oe$ et les applications $\Oe$-linéaires $H$-équivariantes forment une catégorie abélienne (voir \cite[Proposition 2.4.11]{Em1}).

\subsection*{Extensions modulo $p$ et $p$-adiques}

Soient $U$ et $V$ des représentations continues unitaires admissibles de $H$ sur $E$. On note $U_0$ (resp. $V_0$) une boule de $U$ (resp. $V$) stable par $H$ et $U_k$ (resp. $V_k$) sa réduction modulo $\pe^k$ pour tout entier $k \geq 1$.
Le $\Oe$-module $\Hom_H(U_0,V_0)$ s'identifie à un réseau de $\Hom_H(U,V)$ et on a un isomorphisme $\Oe$-linéaire
\begin{equation*}
\Hom_H(U_0,V_0) \cong \varprojlim \Hom_H(U_k,V_k),
\end{equation*}
d'où un isomorphisme $E$-linéaire
\begin{equation} \label{homprojlim}
\Hom_H(U,V) \cong E \otimes_{\Oe} \varprojlim \Hom_H(U_k,V_k).
\end{equation}
De plus pour tout entier $k \geq 1$, le noyau de l'application de réduction
\begin{equation*}
\Hom_H(U_0,V_0) \to \Hom_H(U_k,V_k)
\end{equation*}
est le sous-$\Oe$-module $\pe^k \Hom_H(U_0,V_0)$, d'où avec $k=1$ l'inégalité
\begin{equation*}
\dim_E \Hom_H(U,V) \leq \dim_{\ke} \Hom_H(U_1,V_1).
\end{equation*}
Nous montrons des propriétés analogues (déjà connues, voir l'appendice de \cite{BH}) pour le foncteur $\Ext_H^1$. On calcule $\Ext_H^1(U,V)$ (resp. $\Ext_H^1(U_0,V_0)$, $\Ext_H^1(U_k,V_k)$ avec $k \geq 1$ entier) dans la catégorie des représentations continues unitaires (resp. $\pe$-adiquement continues, lisses) admissibles de $H$ sur $E$ (resp. $\Oe$, $\A{k}$) en utilisant les extensions de Yoneda.

\begin{propapp} \label{prop:extprojlim}
Avec les notations précédentes, si $U$ est résiduellement de longueur finie\footnote{C'est-à-dire que $U_0/\pe U_0$ est de longueur finie (ceci ne dépend pas du choix de $U_0$).}, alors on a un isomorphisme $E$-linéaire
\begin{equation*}
\Ext_H^1(U,V) \cong E \otimes_{\Oe} \varprojlim \Ext_H^1(U_k,V_k)
\end{equation*}
et
\begin{equation*}
\dim_E \Ext_H^1(U,V) \leq \dim_{\ke} \Ext_H^1(U_1,V_1).
\end{equation*}
\end{propapp}

\begin{proof}
Le produit tensoriel par $E$ sur $\Oe$ est exact (en tant que localisation), donc il induit une application $\Oe$-linéaire
\begin{equation*}
\Lambda : \Ext_H^1(U_0,V_0) \to \Ext_H^1(U,V).
\end{equation*}

Montrons que $\ker \Lambda$ est le sous-$\Oe$-module de torsion de $\Ext^1_H(U_0,V_0)$. Comme $\Ext^1_H(U,V)$ est sans torsion, il suffit de montrer que toute classe dans $\ker \Lambda$ est de torsion. Soit donc $\E \in \ker \Lambda$ représentée par une extension
\begin{equation*}
0 \to V_0 \overset{\iota_0}{\longrightarrow} W_0 \overset{\pi_0}{\longrightarrow} U_0 \to 0.
\end{equation*}
Par hypothèse, il existe une section $\sigma : U \to W$ avec $W = E \otimes_{\Oe} W_0$. Comme $U_0$ et $V_0$ sont sans torsion il en est de même pour $W_0$, d'où une injection $W_0 \hookrightarrow W$. Soit $n \in \Nbb$. On peut choisir $n$ suffisamment grand pour que $\pe^n \sigma (U_0) \subset W_0$. Alors $\pe^n\sigma$ induit une section de l'extension
\begin{equation*}
0 \to V_0 \overset{\iota_0}{\longrightarrow} \iota_0(V_0)+\pe^nW_0 \overset{\pe^{-n}\pi_0}{\longrightarrow} U_0 \to 0.
\end{equation*}
Or la classe de cette extension est $\pe^n \E$. Donc $\pe^n \E=0$ et $\E$ est de torsion.

Montrons que $\im \Lambda$ est un réseau de $\Ext_H^1(U,V)$. Il suffit de montrer que toute classe de $\Ext_H^1(U,V)$ est envoyée dans $\im \Lambda$ par une homothétie. Soit donc $\E \in \Ext_H^1(U,V)$ représentée par une extension
\begin{equation*}
0 \to V \overset{\iota}{\longrightarrow} W \overset{\pi}{\longrightarrow} U \to 0.
\end{equation*}
Soit $W_0$ une boule de $W$ stable par $H$. On peut supposer $\iota^{-1}(W_0)=V_0$, quitte à remplacer $W_0$ par $\iota(V_0) + \pe^nW_0$ avec $n \in \Nbb$ suffisamment grand pour que $\iota^{-1}(\pe^nW_0) \subset V_0$.
Soit $n \in \Nbb$. La classe $\pe^n \E$ est représentée par l'extension
\begin{equation*}
0 \to V \overset{\iota}{\longrightarrow} W \overset{\pe^{-n}\pi}{\longrightarrow} U \to 0.
\end{equation*}
On peut choisir $n$ suffisamment grand pour que $U_0 \subset \pe^{-n}\pi(W_0)$. Alors $W_0 \cap (\pe^{-n} \pi)^{-1}(U_0)$ est une boule de $W$ stable par $H$ et par construction c'est une extension de $U_0$ par $V_0$. Ainsi $\pe^n \E \in \im \Lambda$.

Le $\Oe$-module $U_0$ est plat car sans torsion. Pour tout entier $k \geq 1$, la réduction modulo $\pe^k$ des extensions induit donc une application $\Oe$-linéaire $\upsilon_k : \Ext_H^1(U_0,V_0) \to \Ext_H^1(U_k,V_k)$, d'où une application $\Oe$-linéaire
\begin{equation*}
\Upsilon : \Ext_H^1(U_0,V_0) \to \varprojlim \Ext_H^1(U_k,V_k)
\end{equation*}
où les morphismes de transition du système projectif sont les applications $\Oe$-linéaires $\upsilon_{i,j} : \Ext_H^1(U_j,V_j) \to \Ext_H^1(U_i,V_i)$ induites par la réduction modulo $\pe^i$ des extensions pour tous $1 \leq i \leq j$ entiers (le $\A{j}$-module $U_j$ étant plat par changement de base).

Montrons que $\Upsilon$ est surjective. Soit $(\E_k) \in \varprojlim \Ext_H^1(U_k,V_k)$ représentée par des extensions
\begin{equation*}
(0 \to V_k \overset{\iota_k}{\longrightarrow} W_k \overset{\pi_k}{\longrightarrow} U_k \to 0).
\end{equation*}
Ces extensions forment un système projectif qui vérifie la condition de Mittag-Leffler (car les morphismes de transition $V_j \to V_i$ sont surjectifs pour tous $1 \leq i \leq j$ entiers). En passant à la limite projective et en notant $W_0 = \varprojlim W_k$, on obtient donc une suite exacte courte de représentations de $H$ sur $\Oe$
\begin{equation*}
0 \to V_0 \overset{\iota_0}{\longrightarrow} W_0 \overset{\pi_0}{\longrightarrow} U_0 \to 0.
\end{equation*}
Comme $W_0$ est $\pe$-adiquement séparé et complet (en tant que limite projective de $\Oe$-modules $\pe$-adiquement discrets) et sans torsion (puisque $V_0$ et $U_0$ le sont), on obtient une classe $\E_0 \in \Ext_H^1(U_0,V_0)$. Par construction, on a $\Upsilon(\E_0) = (\E_k)$.

Montrons que $\Upsilon$ est injective. Soit $\E_0 \in \ker \Upsilon$ représentée par l'extension
\begin{equation*}
0 \to V_0 \overset{\iota_0}{\longrightarrow} W_0 \overset{\pi_0}{\longrightarrow} U_0 \to 0.
\end{equation*}
Par hypothèse, il existe une section $U_k \to W_k$ avec $W_k = W_0/\pe^k W_0$ pour tout entier $k \geq 1$.
Comme $U_k$ est de type fini sur $A[H]$ (car $U$ est résiduellement de longueur finie) et $V_k$ est admissible (car $V$ est admissible), $\Hom_H(U_k,V_k)$ est de type fini sur $\A{k}$ d'après \cite[Lemme 2.3.10]{Em1} donc fini.
Or, deux sections $U_k \to W_k$ diffèrent par un élément de $\Hom_H(U_k,V_k)$. On en déduit qu'elles sont en nombre fini. On peut donc choisir par récurrence une section $\sigma_k : U_k \to W_k$ pour tout entier $k \geq 1$, de sorte que $\sigma_j \equiv \sigma_i \pmod{\pe^i}$ pour tous $1 \leq i \leq j$ entiers. On obtient ainsi une section $\sigma_0 : U_0 \to V_0$, d'où $\E_0=0$.

On a ainsi montré que $\Upsilon$ est un isomorphisme. Par ce qui précède, on a une application $\Oe$-linéaire
\begin{equation*}
\Lambda \circ \Upsilon^{-1} : \varprojlim \Ext_H^1(U_k,V_k) \to \Ext_H^1(U,V)
\end{equation*}
dont le noyau est le sous-$\Oe$-module de torsion et dont l'image est un réseau. En tensorisant par $E$ sur $\Oe$, on obtient donc l'isomorphisme de l'énoncé.

Il reste à démontrer l'inégalité de l'énoncé. Montrons que pour tout entier $k \geq 1$, le noyau de l'application de réduction
\begin{equation*}
\upsilon_k : \Ext_H^1(U_0,V_0) \to \Ext_H^1(U_k,V_k)
\end{equation*}
est le sous-$\Oe$-module $\pe^k \Ext_H^1(U_0,V_0)$.
Comme $\Ext_H^1(U_k,V_k)$ est un module de $\pe^k$-torsion, il suffit de montrer l'inclusion $\ker \upsilon_k \subset \pe^k \Ext_H^1(U_0,V_0)$. Soit donc $\E_0 \in \ker \upsilon_k$ représentée par l'extension
\begin{equation*}
0 \to V_0 \overset{\iota_0}{\longrightarrow} W_0 \overset{\pi_0}{\longrightarrow} U_0 \to 0.
\end{equation*}
Par hypothèse, il existe une section $W_k \to V_k$ avec $W_k = W_0/\pe^k W_0$. On note $\widetilde{\sigma}_k : W_0 \to V_k$ sa composée avec la projection $W_0 \twoheadrightarrow W_k$. Alors $\E_0 = \pe^k \E_0'$ avec $\E_0' \in \Ext_H^1(U_0,V_0)$ la classe de l'extension
\begin{equation*}
0 \to V_0 \overset{\pe^k\iota_0}{\longrightarrow} \ker \widetilde{\sigma}_k \overset{\pi_0}{\longrightarrow} U_0 \to 0,
\end{equation*}
d'où $\E_0 \in \pe^k \Ext_H^1(U_0,V_0)$.

On a ainsi montré que pour tout entier $k \geq 1$, la réduction modulo $\pe^k$ des extensions induit une injection $\Oe$-linéaire
\begin{equation*}
\Ext_H^1(U_0,V_0) / \pe^k \Ext_H^1(U_0,V_0) \hookrightarrow \Ext_H^k(U_k,V_k).
\end{equation*}
Avec $k=1$, on en déduit l'inégalité de l'énoncé.
\end{proof}

\bibliographystyle{alpha-fr}
\bibliography{these}

\end{document}